\documentclass[
	english
]{scrartcl}

\usepackage[T1]{fontenc}
\usepackage[utf8]{inputenc}

\usepackage{amsmath,amsfonts,amsthm,amssymb}
\usepackage{bm}
\usepackage[colorlinks,citecolor=blue]{hyperref}
\usepackage{doi,natbib}
\usepackage{cancel}

\usepackage{todonotes}

\usepackage{changes}

\usepackage{enumitem}
\setlist[enumerate]{label=(\alph*)}

\usepackage{preamble}

\usepackage{marginnote}

\usepackage[capitalise,nameinlink]{cleveref}

\crefname{figure}{Figure}{Figures}

\allowdisplaybreaks

\numberwithin{equation}{section}

\newcommand\norm[1]{\Vert#1\Vert}
\newcommand\abs[1]{\left\vert#1\right\vert}

\newcommand{\blambda}{{\boldsymbol{\lambda}}}

\newcommand\N{\mathbb{N}}
\newcommand\R{\mathbb{R}}

\newcommand\dx{\mathrm{d}x}

\newcommand{\sgn}{\operatorname{sgn}}

\DeclareMathOperator*{\esssup}{\operatorname{ess\,sup}}

\newcommand{\weakly}{\rightharpoonup}

\DeclareMathAlphabet{\mathpzc}{OT1}{pzc}{m}{it}

\newcommand\fsub{\widehat{\partial}}

\usepackage{mathtools}
\providecommand\given{\nonscript\;\delimsize|\nonscript\;}
\DeclarePairedDelimiterX\set[1]{\{}{\}}{#1}

\DeclarePairedDelimiterX\seq[1]{\{}{\}}{#1}

\DeclarePairedDelimiter\parens{(}{)}

\newtheorem{theorem}{Theorem}[section]
\newtheorem{lemma}[theorem]{Lemma}

\newtheorem{corollary}[theorem]{Corollary}
\newtheorem{remark}[theorem]{Remark}
\newtheorem{definition}[theorem]{Definition}
\newtheorem{example}[theorem]{Example}

\definecolor{mygreen}{rgb}{0.0,0.7,0.0}

\title{%
	Subdifferentiation of nonconvex sparsity-promoting functionals
	on Lebesgue spaces%
}
\author{%
	Patrick Mehlitz%
	\footnote{%
		Brandenburgische Technische Universit\"at Cottbus--Senftenberg,
		Institute of Mathematics,
		03046 Cottbus,
		Germany,
		\email{mehlitz@b-tu.de},
		\url{https://www.b-tu.de/fg-optimale-steuerung/team/dr-patrick-mehlitz},
		ORCID: 0000-0002-9355-850X%
		}
	\and
	Gerd Wachsmuth%
	\footnote{%
		Brandenburgische Technische Universit\"at Cottbus--Senftenberg,
		Institute of Mathematics,
		03046 Cottbus,
		Germany,
		\email{wachsmuth@b-tu.de},
		\url{https://www.b-tu.de/fg-optimale-steuerung/team/prof-gerd-wachsmuth},
		ORCID: 0000-0002-3098-1503%
		}
	}

\date{July 21, 2021}
\publishers{}

{
	\makeatletter
	\def\and{ and }
	\def\footnote#1{}
	\hypersetup{
		pdftitle={\@title},
		pdfauthor={\@author}
	}
	\makeatother
}

\begin{document}
\maketitle

\begin{abstract}
	 Sparsity-promoting terms are incorporated into the objective functions of optimal
	 control problems in order to ensure that optimal controls vanish on large parts
	 of the underlying domain. Typical candidates for those terms are integral functions
	 on Lebesgue spaces based on the $\ell_p$-metric for $p\in[0,1)$ which are nonconvex
	 as well as non-Lipschitz and, thus, variationally challenging. 
	 In this paper, we derive exact formulas for the Fr\'{e}chet, limiting, and singular
	 subdifferential of these functionals.
	 These generalized derivatives can be used for the derivation of necessary optimality
	 conditions for optimal control problems comprising such sparsity-promoting terms.
\end{abstract}

\begin{keywords}
	Integral functionals, Sparsity-promoting functionals, Subdifferentiation, Variational analysis
\end{keywords}

\begin{msc}
	\mscLink{28B05}, \mscLink{49J52}, \mscLink{49J53}, \mscLink{49K27}
\end{msc}

\section{Introduction}

For a measurable and bounded set $\Omega\subset\R^d$
as well as real numbers $s\in[1,\infty)$ and $p\in[0,1)$, we investigate the functional
$q_{s,p}\colon L^s(\Omega)\to\R$ given by
\begin{equation}\label{eq:definition_of_sparsity_functional_p>0}
	\forall u\in L^s(\Omega)\colon\quad
	q_{s,p}(u)
	:=
	\int_\Omega \abs{u(x)}^p\,\mathrm dx
\end{equation}
for $p\in(0,1)$ and by
\begin{equation}\label{eq:definition_of_sparsity_functional_p=0}
	\forall u\in L^s(\Omega)\colon\quad
	q_{s,0}(u)
	:=
	\int_\Omega \abs{u(x)}_0\,\mathrm dx
\end{equation}
for $p=0$ where we used the mapping $\abs{\cdot}_0\colon\R\to\R$ defined as follows:
\[
	\forall y\in\R\colon\quad
	\abs{y}_0
	:=
	\begin{cases}
		0	&	\text{if $y=0$,}\\
		1	&	\text{otherwise.}
	\end{cases}
\]
We note that $q_{s,0}(u)=\blambda(\{u\neq 0\})$ holds for all $u\in L^s(\Omega)$, i.e.,
$q_{s,0}$ measures the size of the \emph{support} of its argument.
In optimal control, the functional $q_{s,p}$ is popular due to its property to
be \emph{sparsity-promoting}, 
see e.g.\ \cite{CasasWachsmuth2020,ItoKunisch2014,Merino2019,NatemeyerWachsmuth2020,Wachsmuth2019},
i.e., to enforce control functions to be zero on large parts of their domain.
This property of $q_{s,p}$ is induced by the fact that the 
mappings $y\mapsto|y|^p$, $p\in(0,1)$, and $y\mapsto|y|_0$ possess a uniquely determined global
minimizer as well as infinite growth at zero.
Let us underline that the case $p=1$, in which the associated mapping $q_{s,1}$ given as in
\eqref{eq:definition_of_sparsity_functional_p>0} reduces to the (convex) norm of
the space $L^1(\Omega)$, is well-studied in the literature, see e.g.\
\cite{CasasHerzogWachsmuth2012,Stadler2009,VossenMaurer2006,WachsmuthWachsmuth2011}.
Clearly, $q_{s,p}$ is not convex for $p\in[0,1)$.

This paper is devoted to the computation of generalized derivatives of the mapping $q_{s,p}$.
More precisely, we aim for the derivation of exact formulas for its so-called Fr\'{e}chet,
limiting, and singular subdifferential, see \cite{Mordukhovich2006}, which can be used in order to
characterize local minimizers of optimal control problems involving $q_{s,p}$ within the objective
function.
The investigation of calculus rules for subdifferentials of nonconvex integral functions 
on Lebesgue spaces has been an active topic of research throughout the last decades.
Exemplary, we would like to mention
\cite{Clarke1983,Giner2017,GinerPenot2018,MordukhovichSagara2018} 
where calculus rules were established in situations where the
integrand satisfies Lipschitzianity assumptions.
We note, however,
that these assumptions typically do \emph{not} hold for the integrands of our interest.
In \cite{CorreaHantoutePerezAros2020}, some upper estimates for the subdifferentials of integral
functions with potentially non-Lipschitzian integrand have been obtained.
Finally, we would like to mention the papers \cite{Chieu2009,Penot2011} where 
some emphasis is laid on nonconvex integral functions
on $L^1(\Omega)$ without assuming any Lipschitzianity of the integrand.
However, as far as we can see, the available results from the literature are of
limited practical use for the actual computation of the subdifferentials associated
with $q_{s,p}$. That is why we directly compute the subdifferentials of interest from their
respective definition. Therefore, we distinguish the cases $p=0$ and $p\in(0,1)$ where
slightly different arguments are necessary in order to proceed.

The remainder of the paper is organized as follows: In \cref{sec:preliminaries}, we comment
on the basic notation used in this paper and put some special emphasis on Lebesgue spaces
and the underlying tools of variational analysis. Furthermore, we briefly investigate the
continuity properties of $q_{s,p}$. Finally, we introduce and study the concept of \emph{slowly
decreasing functions} on Lebesgue spaces which will be used to characterize the points where the
Fr\'{e}chet subdifferential of $q_{s,0}$ is nonempty. In \cref{sec:p=0}, we first compute the
Fr\'{e}chet subdifferential of $q_{s,0}$ and then turn our attention to the characterization
of the limiting and singular subdifferential of this functional. 
In a similar way, we proceed in \cref{sec:0<p<1} in order to address the functional $q_{s,p}$ for
$p\in(0,1)$. Some concluding remarks close the paper in \cref{sec:concluding_remarks}.

\section{Preliminaries}\label{sec:preliminaries}

\subsection{Basic notation}

For a sequence $\{z_k\}_{k\in\N}\subset Z$ in a real Banach space $Z$ and some point $z\in Z$,
we exploit $z_k\to z$ ($z_k\weakly z$) in order to denote that $\{z_k\}_{k\in\N}$ converges
strongly (weakly) to $z$. Furthermore, for a sequence $\{t_k\}_{k\in\N}\subset\R$
and $\alpha\in\R$, $t_k\searrow \alpha$
means $\{t_k\}_{k\in\N}\subset(\alpha,\infty)$ and $t_k\to \alpha$. Similarly, we interpret
$t_k\nearrow \alpha$.

Whenever a function $f\colon Z\to\R$ is Fr\'{e}chet differentiable at $z\in Z$, its Fr\'{e}chet
derivative will  be denoted by $f'(z)\in Z^*$ where $Z^*$ is the (topological) dual of $Z$.

\subsection{Lebesgue spaces}

Throughout this paper,
we assume that
$\Omega\subset\R^d$ is Lebesgue-measurable
with positive and finite Lebesgue measure.
We equip $\Omega$ with the $\sigma$-algebra of all Lebesgue-measurable subsets
of $\Omega$ as well as Lebesgue's measure $\blambda$.
For brevity, we will suppress the prefix \emph{Lebesgue} in the course of the manuscript.
The characteristic function of a measurable set $A\subset\Omega$, being $1$ on $A$ while vanishing on
$\Omega\setminus A$, will be denoted by $\chi_A\colon\Omega\to\{0,1\}$.
Whenever $u\colon\Omega\to\R$ is measurable, the associated function $\sgn u\colon \Omega\to\set{-1,0,1}$,
which assigns to each $x\in\Omega$ the sign of $u(x)$, is measurable as well.
For $s\in[1,\infty)$, we use the classical Lebesgue spaces $L^s(\Omega)$ of all
(equivalence classes of)
real-valued, measurable, $s$-integrable functions which are equipped with the classical norm
\[
	\forall u\in L^s(\Omega)\colon\quad
	\norm{u}_s:=\left(\int_\Omega|u(x)|^s\,\mathrm dx\right)^{1/s}.
\]
For the purpose of completeness, let us recall that $L^\infty(\Omega)$ comprises all (equivalence
classes of) real-valued, measurable functions which are essentially bounded, and that this set becomes
a Banach space when equipped with the norm
\[
	\forall u\in L^\infty(\Omega)\colon\qquad
	\norm{u}_\infty:=\esssup\limits_{x\in\Omega}|u(x)|.
\]
Fixing $s\in[1,\infty]$, for each function $u\in L^s(\Omega)$, we use
\[
	\{u\neq 0\}:=\{x\in\Omega\,|\,u(x)\neq 0\}
	\qquad\text{and}\qquad
	\{u= 0\}:=\{x\in\Omega\,|\,u(x)=0\}
\]
for brevity of notation
and note that these sets are measurable and well-defined up to sets of measure zero.
Similarly, we define the measurable sets $\{u\geq 0\}$, $\{u>0\}$, $\{u\leq 0\}$, and $\{u<0\}$
as well as sets with non-zero or bilateral bounds.
Throughout the paper, $r\in(1,\infty]$ given by $\tfrac1s+\tfrac1r=1$ is the conjugate
coefficient associated with $s\in[1,\infty)$. It is well known that $L^r(\Omega)$ can be identified with the (topological) dual space
of $L^s(\Omega)$. 
Finally, if $\Omega'\subset\Omega$ is measurable, we exploit the notation
\[
	\forall u\in L^s(\Omega),\,\forall \nu\in(0,s]\colon\quad
	\norm{u}_{\nu,\Omega'}:=\left(\int_{\Omega'}|u(x)|^\nu\,\mathrm dx\right)^{1/\nu}
\]
and note that $\norm{u}_{\nu,\Omega'}$ is a finite number for each $u\in L^s(\Omega)$ and each
$\nu\in(0,s]$.

We briefly mention that $\Omega \subset \R^d$
can be replaced by an arbitrary measure space $(\Omega,\Sigma,\mu)$
such that $\mu$ is finite, separable, and non-atomic.
Indeed, such a measure space is isomorphic to
the Lebesgue measure on some interval,
see, e.g., \cite[Theorem~9.3.4]{Bogachev2007}.

\subsection{Tools from variational analysis}\label{sec:variational_analysis}

In this paper, we are concerned with the computation of subdifferentials 
associated with the functional $q_{s,p}$
with $s\in[1,\infty)$ and $p\in[0,1)$.
Recall that, for a given point $\bar u\in L^s(\Omega)$, the so-called Fr\'{e}chet (or regular)
subdifferential of $q_{s,p}$ at $\bar u$ is defined by means of
\begin{equation*}
	\fsub q_{s,p}(\bar u)
	:=
	\set*{
		\eta \in L^r(\Omega)
		\given
		\liminf_{\norm{h}_s \searrow 0}
		\frac{
			q_{s,p}(\bar u + h) - q_{s,p}(\bar u) - \int_\Omega \eta(x)h(x)\,\mathrm dx
		}{
			\norm{h}_s
		}
		\ge 0
	}
	.
\end{equation*}
Furthermore, in case $s>1$, the limiting (or Mordukhovich) and the singular
subdifferential of $q_{s,p}$
at $\bar u$ are defined as stated below:
\begin{align*}
	\partial q_{s,p}(\bar u)
	&:=
	\left\{\eta\in L^r(\Omega)\,\middle|\,
		\begin{aligned}
			&\exists\{u_k\}_{k\in\N}\subset L^s(\Omega),\,
				\exists\{\eta_k\}_{k\in\N}\subset L^r(\Omega)
				\colon\\
			&\quad u_k\to\bar u\text{ in }L^s(\Omega),\,
				q_{s,p}(u_k)\to q_{s,p}(\bar u),\\
			&\quad
				\eta_k\weakly\eta\text{ in }L^r(\Omega),\eta_k\in\fsub q_{s,p}(u_k)\,\forall k\in\N
		\end{aligned}
	\right\},
	\\
	\partial^\infty q_{s,p}(\bar u)
	&:=
	\left\{\eta\in L^r(\Omega)\,\middle|\,
		\begin{aligned}
			&\exists\{u_k\}_{k\in\N}\subset L^s(\Omega),\,
				\exists\{t_k\}_{k\in\N}\subset(0,\infty),\\
			&\exists\{\eta_k\}_{k\in\N}\subset L^r(\Omega)
				\colon\\
			&\quad u_k\to\bar u\text{ in }L^s(\Omega),\,
				q_{s,p}(u_k)\to q_{s,p}(\bar u),\,t_k\searrow 0,\,\\
			&\quad
				t_k\eta_k\weakly\eta\text{ in }L^r(\Omega),\,\eta_k\in\fsub q_{s,p}(u_k)\,\forall k\in\N
		\end{aligned}
	\right\}
	.
\end{align*}
Noting that $L^1(\Omega)$ is not a so-called Asplund space,
i.e., a Banach space where every convex, continuous
functional is \emph{generically} Fr\'{e}chet differentiable, one cannot simply define the limiting
and singular subdifferential of $q_{1,p}$ as a set-limit of the associated Fr\'{e}chet subdifferential
while preserving its variational properties.
Instead, the larger so-called $\varepsilon$-subdifferential of $q_{1,p}$ would be needed
within the limiting procedure. 
We would like to point out that working with the limiting variational tools in spaces
which do not possess the Asplund property has been shown to be problematic.
A detailed discussion can be found in \cite[Section~2.2]{Mordukhovich2006}.
Nevertheless, due to \cite[Theorem~3.2]{Chieu2009}, we have
$\partial q_{1,p}(\bar u)=\fsub q_{1,p}(\bar u)$ for all $\bar u\in L^1(\Omega)$ 
and all $p\in[0,1)$ 
(with respect to the correct definition of the limiting subdifferential in non-Asplund spaces), 
and, thus, our results from
\cref{thm:characterization_Frechet_q10,thm:characterization_Frechet_q1p}
yield explicit formulas for the limiting subdifferential of $q_{1,p}$ as well.
In variational analysis, the limiting subdifferential has turned out to be a valuable tool
for the derivation of necessary optimality conditions for constrained optimization problems,
see \cite[Section~5]{Mordukhovich2006}.
On the other hand, the singular subdifferential provides a measure of Lipschitzianity of nonsmooth functionals. Applied in the context of this paper, 
thanks to \cite[Corollary~2.39, Theorem~3.52]{Mordukhovich2006} and
the continuity properties of $q_{s,p}$, see \cref{sec:continuity}, we have the following result.
\begin{lemma}\label{lem:Lipschitzness_s>1}
	Fix $s\in(1,\infty)$ and $p\in[0,1)$.
	Then $q_{s,p}$ is Lipschitz continuous at some point $\bar u\in L^s(\Omega)$ if and only if
	the following conditions hold:
	\begin{enumerate}
		\item we have $\partial^\infty q_{s,p}(\bar u)=\{0\}$ and
		\item\label{item:SNEC}
			 for sequences $\{u_k\}_{k\in\N}\subset L^s(\Omega)$, $\{t_k\}_{k\in\N}\subset(0,\infty)$,
			and $\{\eta_k\}_{k\in\N}\subset L^r(\Omega)$ satisfying $u_k\to\bar u$ in $L^s(\Omega)$,
			$q_{s,p}(u_k)\to q_{s,p}(\bar u)$,
			$t_k\searrow 0$, $t_k\eta_k\weakly 0$ in $L^r(\Omega)$, and $\eta_k\in\fsub q_{s,p}(u_k)$
			for each $k\in\N$, we already have $t_k\eta_k\to 0$ in $L^r(\Omega)$.
	\end{enumerate}
\end{lemma}

Let us note that the property from \cref{lem:Lipschitzness_s>1}~\ref{item:SNEC} is referred to
as \emph{sequential normal epi-compactness} of $q_{s,p}$ at $\bar u$ in variational analysis.
The latter is related to the so-called \emph{sequential normal compactness} property of sets which
has been shown to be problematic in Lebesgue spaces, see \cite[Section~4]{Mehlitz2019}.

\subsection{Continuity properties of sparsity-promoting functionals}\label{sec:continuity}

In this section, we briefly comment on the continuity properties of the sparsity-promoting functional
$q_{s,p}$ for $s\in[1,\infty)$ and $p\in[0,1)$.
We split our investigation into two lemmas.
\begin{lemma}\label{lem:lower_semicontinuity_of_qs0}
	For each $s\in[1,\infty)$, $q_{s,0}$ is lower semicontinuous.
\end{lemma}
\begin{proof}
	Fix a sequence $\{u_k\}_{k\in\N}\subset L^s(\Omega)$ converging to some $\bar u\in L^s(\Omega)$.
	For subsequent use, we set $\alpha:=\liminf_{k\to\infty}q_{s,0}(u_k)$.
	We pick a subsequence (without relabeling) with $q_{s,0}(u_k)\to\alpha$ and assume 
	without loss of generality (w.l.o.g.) that
	$\{u_k\}_{k\in\N}$ converges pointwise almost everywhere to $\bar u$
	along this subsequence.	
	Noting that $y\mapsto|y|_0$ is lower semicontinuous, we find
	\begin{align*}
		\alpha
		=
		\lim\limits_{k\to\infty} q_{s,0}(u_k)
		&=
		\lim\limits_{k\to\infty}\int_\Omega|u_k(x)|_0\,\mathrm dx\\
		&\geq
		\int_\Omega\left(\liminf\limits_{k\to\infty}|u_k(x)|_0\right)\,\mathrm dx
		\geq
		\int_\Omega|\bar u(x)|_0\,\mathrm dx
		=
		q_{s,0}(\bar u)
	\end{align*}
	from Fatou's lemma, and this shows the claim.
\end{proof}
It is also clear that $q_{s,0}$ is not continuous
at points $\bar u \in L^s(\Omega)$ with $\blambda(\{\bar u= 0\}) > 0$.
\begin{lemma}\label{lem:uniform_continuity_of_qsp}
	For each $s\in[1,\infty)$ and $p\in(0,1)$, $q_{s,p}$ is uniformly continuous.
\end{lemma}
\begin{proof}
	Noting that $\varphi\colon\R\to\R$ given by $\varphi(y):= |y|^p$ for all
	$y\in\R$ is continuous and that the associated Nemytskii operator
	maps $L^s(\Omega)$ into $L^1(\Omega)$, the latter is continuous due to
	\cite[Theorem~4]{GoldbergKampowskyTroeltzsch1992}. 
	Furthermore, integration of functions is a linear, continuous operation on $L^1(\Omega)$.
	Thus, $q_{s,p}$ is the composition of two continuous mappings and, thus, continuous.
	
	We note that $\varphi$ is subadditive, i.e., $|y_1+y_2|^p\leq|y_1|^p+|y_2|^p$ holds
	for all $y_1,y_2\in\R$. Applying this inequality first to $y_1:=z_1-z_2$ as well as $y_2:=z_2$ 
	and second to $y_1:=z_2-z_1$ as well as $y_2:=z_1$ for $z_1,z_2\in\R$ yields the estimate
	\[
		\forall z_1,z_2\in\R\colon\quad
		\bigl| |z_1|^p-|z_2|^p\bigr|\leq |z_1-z_2|^p.
	\]
	
	Fix an arbitrary $\varepsilon>0$. By continuity of $q_{s,p}$ at $0$, we find $\delta>0$ such that
	$q_{s,p}(u)<\varepsilon$ holds for all $u\in L^s(\Omega)$ such that $\norm{u}_s<\delta$.
	Thus, for any two functions $u,v\in L^s(\Omega)$ satisfying $\norm{u-v}_s<\delta$, we find
	\begin{align*}
		|q_{s,p}(u)-q_{s,p}(v)|
		&=
		\left|\int_\Omega\bigl(|u(x)|^p-|v(x)|^p\bigr)\,\mathrm dx\right|
		\leq
		\int_\Omega\bigl||u(x)|^p-|v(x)|^p\bigr|\,\mathrm dx\\
		&\leq
		\int_\Omega |u(x)-v(x)|^p\,\mathrm dx
		=
		q_{s,p}(u-v)
		<
		\varepsilon,
	\end{align*}
	and this yields uniform continuity of $q_{s,p}$.
\end{proof}

The inherent nonconvexity of the functional $q_{s,p}$ indicates that it is not weakly lower
semicontinuous. Thus, one has to face essential issues regarding the existence of solutions
whenever $q_{s,p}$ is used as an additional sparsity-promoting term in the objective function
of an optimal control problem where the controls are chosen from a Lebesgue space,
see \cite{ItoKunisch2014,Wachsmuth2019} where this is discussed in detail.

\subsection{Slowly decreasing functions}\label{sec:s_SD_functions}

Fix $s\in(1,\infty)$.
In the course of the paper, it will become clear that the Fr\'{e}chet subdifferential
of $q_{s,0}$ at some point $\bar u\in L^s(\Omega)$ such that $\{\bar u\neq 0\}$ is
of positive measure is likely to be empty if $\bar u$ approaches zero on $\{\bar u\neq 0\}$
too fast. In this regard, the following definition aims to characterize functions tending to zero
slowly enough on their support. 
\begin{definition}\label{def:slowly_decrasing_functions}
	Fix $s\in(1,\infty)$ as well as a function $\bar u\in L^s(\Omega)$.
	We call $\bar u$ \emph{order $s$ slowly decreasing}
	(an \emph{$s$-SD function} for short) whenever
	for each sequence $\{\Omega_k\}_{k\in\N}$ of
	measurable subsets of $\{\bar u\neq 0\}$, we have
	\[
		\blambda(\Omega_k)\searrow 0
		\qquad\Longrightarrow\qquad
		\frac{\blambda(\Omega_k)}{\norm{\bar u}_{s,\Omega_k}}
		\searrow 0.
	\] 
\end{definition}
Note that whenever $\bar u\in L^s(\Omega)$ vanishes almost everywhere
on $\Omega$, then it is trivially $s$-SD
since there are no measurable subsets $\Omega_k$ of $\set{\bar u \ne 0}$
with positive measure.

In the subsequently stated lemma, we present a sufficient condition implying that
a given function $\bar u\in L^s(\Omega)$ is an $s$-SD function. 
Its proof follows straight from the definition and is, therefore, omitted.
\begin{lemma}\label{lem:slowly_decreasing_functions}
	Fix $s\in(1,\infty)$ and $\bar u\in L^s(\Omega)$.
	Assume that $\bar u$ is bounded away from zero on $\{\bar u\neq 0\}$, i.e., 
	that one can find $\varepsilon>0$ such that $\blambda(\{0<\abs{\bar u}<\varepsilon\})=0$ is valid.
	Then $\bar u$ is an $s$-SD function.
\end{lemma}

The following example shows that the condition from 
\cref{lem:slowly_decreasing_functions}
is only sufficient but not necessary for the property of $\bar u\in L^s(\Omega)$
to be an $s$-SD function. 
\begin{example}\label{ex:condition_nonemptiness_Frechet_qs0}
	Consider $\Omega:=(0,1)$ and the function
	$\bar u\in L^s(\Omega)$, $s>1$, given by $\bar u(x):=x^\alpha$ for each $x\in\Omega$ 
	and some $\alpha>0$.
	We want to check for which choices of $s$ and $\alpha$, $\bar u$
	actually is an $s$-SD function. Considering sets $\Omega_k\subset\Omega$ 
	of positive measure such that
	$\blambda(\Omega_k)$ is fixed, the quotient $\blambda(\Omega_k)/\norm{\bar u}_{s,\Omega_k}$
	gets maximal whenever $\bar u$ is as small as possible on $\Omega_k$. 
	Thus, by strict monotonicity of $\bar u$, it suffices to consider sequences of sets
	$\{\Omega_k\}_{k\in\N}$ of the form $\Omega_k:=(0,t_k)$ where $\{t_k\}_{k\in\N}$
	satisfies $t_k\searrow 0$.
	In this case, we obtain
	\begin{align*}
		\frac{\blambda(\Omega_k)}{\norm{\bar u}_{s,\Omega_k}}
		=
		(\alpha s+1)^{1/s}\,t_k^{1-\alpha-1/s},
	\end{align*}
	and this shows that $\bar u$ is an
	$s$-SD function if and only if $\alpha+1/s<1$ holds true.
	Particularly, $\alpha<1$ is necessary, and, in this case, $\bar u$ tends to $0$ quite
	slowly.
	Observe that $\alpha<1-1/s$ is equivalent to $\abs{\bar u}^{-1}\in L^r(\Omega)$.
	Recall that $r \in (1,\infty)$ is the conjugate coefficient associated with $s$.
\end{example}

In the remainder of the section, we aim to find a more tractable characterization of
$s$-SD functions. The above example motivates the subsequently stated lemma.

\begin{lemma}
	\label{lem:SD_by_hoelder}
	Let $s \in (1,\infty)$ be given.
	Furthermore, fix a function $\bar u \in L^s(\Omega)$ 
	with
	$\abs{\bar u}^{-1} \chi_{\set{\bar u \ne 0}} \in L^r(\Omega)$.
	Then $\bar u$ is an $s$-SD function.
\end{lemma}
\begin{proof}
	Pick an arbitrary sequence $\{\Omega_k\}_{k\in\N}$ of measurable
	subsets of $\{\bar u\neq 0\}$ satisfying $\blambda(\Omega_k)\searrow 0$.
	For each $k\in\N$, we find
	\begin{equation*}
		\blambda(\Omega_k)
		=
		\int_{\Omega_k} \abs{\bar u(x)} \abs{\bar u(x)}^{-1} \, \dx
		\le
		\norm{\bar u}_{s,\Omega_k} \norm{\abs{\bar u}^{-1}}_{r,\Omega_k}
	\end{equation*}
	by applying H\"older's inequality on $\Omega_k$.
	From
	$\abs{\bar u}^{-1} \chi_{\set{\bar u \ne 0}} \in L^r(\Omega)$
	we find the convergence 
	$\norm{\abs{\bar u}^{-1}}_{r,\Omega_k} \searrow 0$
	since $\blambda(\Omega_k) \searrow 0$.
	Hence, $\bar u$ is an $s$-SD function.
\end{proof}

Note that the requirements of \cref{lem:slowly_decreasing_functions} are
sufficient for the ones of \cref{lem:SD_by_hoelder}.
The next example shows the existence of $s$-SD functions $\bar u$
for which $\abs{\bar u}^{-1}\chi_{\{\bar u\neq 0\}} \not\in L^r(\Omega)$ holds.
\begin{example}
	\label{ex:SD_not_Lr}
	Again we consider $\Omega := (0,1)$.
	Let $s \in (1,\infty)$ be arbitrary.
	For each $k\in\N$, we define
	$t_k := 2^{-k}$.
	For some monotonically decreasing sequence
	$\seq{\gamma_k}_{k\in\N} \subset (0,\infty)$
	satisfying $\sum_{j=1}^\infty\gamma_j^s2^{-j}<\infty$,
	we consider the (monotonically increasing) function
	\begin{align*}
		\bar u &:= \sum_{j = 1}^\infty \gamma_j \chi_{(t_{j+1}, t_j]}
		\in L^s(\Omega)
		.
	\end{align*}
	Furthermore, we make use of $\Omega_k:=(0,t_k)$ for each $k\in\N$
	and note that $\blambda(\Omega_k)=t_k$ is valid.
	We obtain the estimates
	\begin{align*}
		\gamma_k^s t_k
		\ge
		\norm{\bar u}_{s,\Omega_k}^s
		&=
		\sum_{j = k}^\infty \gamma_j^s \, \parens{t_j - t_{j+1}}
		\ge
		\gamma_k^s \, \parens{t_k - t_{k+1}}
		=
		2^{-1} \gamma_k^s t_k
		,
		\\
		\gamma_k^{- 1} t_k^{1 - 1/s}
		\le
		\frac{\blambda(\Omega_k)}{\norm{\bar u}_{s,\Omega_k}}
		&
		\le
		2^{1/s} \gamma_k^{- 1} t_k^{1 - 1/s}.
	\end{align*}
	This shows that $\bar u$ is $s$-SD
	only if
	$\gamma_k^{-r} t_k \to 0$.
	Actually, this is already sufficient for the $s$-SD property. 
	Indeed, by monotonicity of $\bar u$, is suffices to consider sequences $\{\Omega_k'\}_{k\in\N}$
	of type $\Omega_k':=(0,t_k')$ for sequences $\{t_k'\}_{k\in\N}\subset(0,1)$ 
	satisfying $t_k'\searrow 0$.
	Then, for each $k\in\N$, we find $\ell_k\in\N$ such that $t_{\ell_k+1}\leq t_k'<t_{\ell_k}$
	leading to $\ell_k\to\infty$, $\blambda(\Omega_k')< t_{\ell_k}=2t_{\ell_k+1}$, and 
	\begin{align*}
		\norm{\bar u}_{s,\Omega_k'}^s
		=
		\sum\limits_{j=\ell_k+1}^\infty \gamma_j^s(t_j-t_{j+1})+\gamma_{\ell_k}^s(t_k'-t_{\ell_k+1})
		\geq
		2^{-1}\gamma_{\ell_k+1}^st_{\ell_k+1}.
	\end{align*}
	This yields
	\begin{align*}
		\frac{\blambda(\Omega_k')}{\norm{\bar u}_{s,\Omega_k'}}
		\leq
		\frac{2t_{\ell_k+1}}{2^{-1/s}\gamma_{\ell_k+1}t_{\ell_k+1}^{1/s}}
		=
		2^{1+1/s}\gamma_{\ell_k+1}^{-1}t_{\ell_k+1}^{1-1/s}
		\to
		0
	\end{align*}
	which shows validity of the $s$-SD property if $\gamma_k^{-r}t_k\to 0$ is valid.

	In particular,
	choosing
	$\gamma_k := (k 2^{-k})^{1/r}$
	for each $k\in\N$,
	the function $\bar u$ is $s$-SD,
	but
	\begin{equation*}
		\norm{ \abs{\bar u}^{-1} }_{r}^r
		=
		\sum_{j = 1}^\infty \gamma_j^{-r} \parens{t_j - t_{j+1}}
		=
		\sum_{j = 1}^\infty j^{-1} 2^j 2^{-j-1}
		=
		\infty,
	\end{equation*}
	i.e.,
	the sufficient condition from \cref{lem:SD_by_hoelder}
	does not hold.
\end{example}
In \cref{ex:SD_not_Lr},
the coupling between $t_k$ and $\gamma_k$
is
decisive
for the $s$-SD property.
This will be made more precise in \cref{thm:weak_sd_Hugo}.
For the proof of it, we need an auxiliary lemma,
which provides an intermediate value theorem for monotonic functions.
\begin{lemma}
	\label{lem:IVT}
	Let $g \colon [0,\infty) \to [0,\infty)$ be monotonically increasing
	and not identically $0$.
	Moreover, let $\alpha>0$ be fixed.
	Then, for each $C > 0$, there exists a unique $\gamma_C > 0$
	such that
	\begin{equation*}
		\lim_{\gamma \nearrow \gamma_C} g(\gamma)
		\le
		\frac{C}{\gamma_C^\alpha}
		\le
		\lim_{\gamma \searrow \gamma_C} g(\gamma)
		.
	\end{equation*}
	If $C \searrow 0$ we have $\gamma_C \searrow 0$.
\end{lemma}
\begin{proof}
	Observing that $\gamma\mapsto \gamma^\alpha g(\gamma)$ 
	is monotonically increasing on $[0,\infty)$,
	one can check that
	\begin{equation*}
		\gamma_C
		:=
		\inf\set{ \gamma \in (0,\infty) \given g(\gamma) \ge C \gamma^{-\alpha}}
		=
		\sup\set{ \gamma \in (0,\infty) \given g(\gamma) < C \gamma^{-\alpha}}
	\end{equation*}
	satisfies the desired inequalities.
	The uniqueness follows since
	$\gamma \mapsto C \gamma^{-\alpha}$
	is strictly monotonically decreasing.
	Moreover,
	if $\varepsilon > 0$ is arbitrary
	and
	$C \le \varepsilon^\alpha g(\varepsilon)$
	holds
	then
	$\varepsilon$ belongs to the set under the infimum
	and, therefore,
	$\gamma_C \le \varepsilon$ follows.
\end{proof}

\begin{theorem}
	\label{thm:weak_sd_Hugo}
	Let $s \in (1,\infty)$ be given.
	Then $\bar u \in L^s(\Omega)$
	is an $s$-SD function
	if and only if
	\begin{equation*}
		\lim\limits_{\gamma\searrow 0}
		\blambda\parens*{
			\set*{
				0 < \abs{\bar u} \le \gamma
			}
		}
		\gamma^{-r}
		=
		0
		.
	\end{equation*}
\end{theorem}
\begin{proof}
	``$\Longrightarrow$'':
	For an arbitrary sequence $\seq{\gamma_k}_{k\in\N} \subset (0,\infty)$ with $\gamma_k \searrow 0$,
	we set $\Omega_k := \set{ 0 < \abs{\bar u} \le \gamma_k}$ for each $k\in\N$.
	Exploiting $\bigcap_{k\in\N}\Omega_k=\varnothing$, 
	one can easily check that $\blambda(\Omega_k)\to 0$ is valid. 
	In case where $\{\blambda(\Omega_k)\}_{k\in\N}$ vanishes along the tail of the
	sequence, we have nothing to show.
	Thus, let us assume $\blambda(\Omega_k)\searrow 0$.
	Then we have
	\begin{equation*}	
		0
		\le 
		\lim\limits_{k\to\infty}
		\blambda(\Omega_k)^{1/r} \gamma_k^{-1}
		=
		\lim\limits_{k\to\infty}
		\blambda(\Omega_k)^{1 - 1/s} \gamma_k^{-1}
		\le 
		\lim\limits_{k\to\infty}
		\frac{\blambda(\Omega_k)}{\norm{\bar u}_{s,\Omega_k}}
		=
		0
	\end{equation*}
	by definition of an $s$-SD function.

	``$\Longleftarrow$'':
	Let $\{\Omega_k\}_{k\in\N}$ be a sequence of measurable subsets of $\set{\bar u \ne 0}$ 
	with $\blambda(\Omega_k) \searrow 0$.
	For an arbitrary $\gamma > 0$ and $k\in\N$,
	we have
	\begin{equation*}
		\Omega_k
		=
		\set{ x \in \Omega_k \given \abs{\bar u} \ge \gamma }
		\cup
		\set{ x \in \Omega_k \given \abs{\bar u} < \gamma }
		.
	\end{equation*}
	Using Chebyshev's inequality,
	we get
	\begin{equation}\label{eq:some_upper_estimate_s_SD}
		\blambda(\Omega_k)
		\le
		\frac{\norm{\bar u}_{s,\Omega_k}^s}{\gamma^s} + \blambda(\set{0 < \abs{\bar u} < \gamma})
		.
	\end{equation}
	In order to equilibrate the addends on the right-hand side,
	we apply \cref{lem:IVT}
	with $g(t) := \blambda(\set{0 < \abs{\bar u} < t})$
	and $\alpha:=s$
	in order to
	obtain $\gamma_k > 0$
	such that
	\begin{equation*}
		\blambda(\set{0 < \abs{\bar u } < \gamma_k })
		\le
		\frac{\norm{\bar u}_{s,\Omega_k}^s}{\gamma_k^{s}}
		\le
		\blambda(\set{0 < \abs{\bar u } \le \gamma_k })
		.
	\end{equation*}
	Due to $\norm{\bar u}_{s,\Omega_k} \searrow 0$,
	\cref{lem:IVT} guarantees $\gamma_k \searrow 0$.
	From \eqref{eq:some_upper_estimate_s_SD}, we infer
	\begin{equation*}
		\blambda(\Omega_k)
		\leq
		2\norm{\bar u}_{s,\Omega_k}^s \gamma_k^{-s}
		\qquad\text{and}\qquad
		\blambda(\Omega_k)
		\leq
		2\blambda(\set{0 < \abs{\bar u } \le \gamma_k }).
	\end{equation*}
	We raise these two inequalities to the powers $r / (r + s)$
	and $s / (r + s)$, respectively,
	and multiply them to obtain
	\begin{equation*}
		\blambda(\Omega_k)
		\le
		2
		\blambda(\set{0 < \abs{\bar u } \le \gamma_k })^{s/(r+s)}
		\gamma_k^{-r s / (r + s)}
		\norm{\bar u}_{s, \Omega_k}^{r s / (r + s)}
		.
	\end{equation*}
	Using $r s / (r + s) = 1$ and $\gamma_k \searrow 0$,
	we get
	\begin{equation*}
		\frac{\blambda(\Omega_k)}{\norm{\bar u}_{s,\Omega_k}}
		\le
		2
		\parens*{
			\blambda(\set{0 < \abs{\bar u } \le \gamma_k })
			\gamma_k^{-r}
		}^{s/(r+s)}
		\to
		0
	\end{equation*}
	which completes the proof.
\end{proof}

Note that the condition from \cref{thm:weak_sd_Hugo} is a little bit stronger than
$\abs{\bar u}^{-1}\chi_{\{\bar u\neq 0\}} \in L^{r,\infty}(\Omega)$,
where $L^{r,\infty}(\Omega)$
is a weak Lebesgue space (or Lorentz space),
which would require that
$\blambda\parens*{
	\set*{
		0 < \abs{\bar u} \le \gamma
	}
}
\gamma^{-r}$
is bounded with respect to $\gamma \in (0,\infty)$.
Moreover,
Chebyshev's inequality can be used to
see that $\abs{\bar u}^{-1}\chi_{\{\bar u\neq 0\}} \in L^r(\Omega)$
implies the condition from \cref{thm:weak_sd_Hugo}.
Indeed,
\begin{equation*}
	\blambda\parens*{
		\set*{
			0 < \abs{\bar u} \le \gamma
		}
	}
	\gamma^{-r}
	\le
	\norm{ \abs{\bar u}^{-1} }_{r, \set{\abs{\bar u}^{-1} \ge \gamma^{-1}}}^r
	\to
	0
	\qquad\text{as }\gamma \searrow 0
\end{equation*}
holds if $\abs{\bar u}^{-1}\chi_{\{\bar u\neq 0\}} \in L^r(\Omega)$ is valid.

\section{The case $p=0$}\label{sec:p=0}

We start our analysis by investigating the variational properties of the 
discontinuous functional $q_{s,0}$. 

\subsection{Fr\'{e}chet subdifferential}

The aim of this subsection is to provide a full characterization of the Fr\'{e}chet subdifferential
associated with the functional $q_{s,0}$ for each $s\in[1,\infty)$.
We start our investigations by providing a simple upper bound of the Fr\'{e}chet subdifferential
of $q_{s,0}$.
\begin{lemma}\label{lem:trivial_upper_bound_Frechet_qs0}
	For given $s\in[1,\infty)$ and $\bar u\in L^s(\Omega)$, we have
	\[
		\widehat{\partial}q_{s,0}(\bar u)
		\subset
		\{\eta\in L^r(\Omega)\,|\,\{\eta\neq 0\}\subset \{\bar u= 0\}\}.
	\]
\end{lemma}
\begin{proof}
	Let $\eta \in L^r(\Omega)$
	be given such that
	there exists a measurable set $\Omega'\subset\{\bar u\neq 0\}$ of non-zero
	measure where $\eta$ is non-vanishing.
	We assume w.l.o.g.\ that $\abs{\eta(x)}\geq\rho$ holds for some $\rho>0$
	and almost all $x\in\Omega'$.
	Define a sequences
	$\{h_k\}_{k\in\N}\subset L^s(\Omega)$ by means of
	$h_k:=\tfrac1{2k}\abs{\bar u}\chi_{\Omega'}\sgn\eta$ for each $k\in\N$.
	Clearly, we have $\norm{h_k}_s\searrow 0$.
	Furthermore, we find
	\begin{align*}
		&
		\frac{q_{s,0}(\bar u+h_k)-q_{s,0}(\bar u)-\int_{\Omega}\eta(x)h_k(x)\,\mathrm dx}
		{\norm{h_k}_s}
		=
		\frac{
			-\tfrac1{2k}\int_{\Omega'}|\eta(x)||\bar u(x)|\,\mathrm dx
		}{	
			\tfrac1{2k}\,\norm{\bar u}_{s,\Omega'}
		}
		\leq
		-\rho\,
		\frac{
		 	\norm{\bar u}_{1,\Omega'}
		 }{
		 	\norm{\bar u}_{s,\Omega'}
		 }
		<
		0.
	\end{align*}
	Hence, $\eta\notin\widehat\partial q_{s,0}(\bar u)$
	and this finishes the proof.
\end{proof}

In the subsequently stated result, we characterize all points in $L^s(\Omega)$ where the Fr\'{e}chet
subdifferential of $q_{s,0}$ is nonempty.
Therefore, the concept of slowly decreasing functions discussed in \cref{sec:s_SD_functions}
turns out to be essential.
\begin{lemma}\label{lem:nonemptiness_Frechet_qs0}
	Fix $s\in[1,\infty)$ and $\bar u\in L^s(\Omega)$.
	Then $\widehat\partial q_{s,0}(\bar u)$ is nonempty if and only if one of the
	following conditions is valid:
	\begin{enumerate}
		\item \label{item:nonemptiness_Frechet_qs0_trivial}$\bar u=0$ holds almost everywhere
			on $\Omega$,
		\item \label{item:nonemptiness_Frechet_qs0} 
			it holds $s>1$ and 
			$\bar u$ is an $s$-SD function.
	\end{enumerate}
\end{lemma}
\begin{proof}
	We start the proof by showing that $0\in\widehat\partial q_{s,0}(\bar u)$ holds in the
	presence of each of the given conditions, i.e., we need to show that for all sequences
	$\{h_k\}_{k\in\N}\subset L^s(\Omega)$ with $\norm{h_k}_s\searrow 0$, we have
	\[
		\liminf\limits_{k\to\infty}
		\frac{q_{s,0}(\bar u+h_k)-q_{s,0}(\bar u)}{\norm{h_k}_s}
		\geq 0.
	\]
	This obviously holds true whenever $q_{s,0}(\bar u)=0$ holds, i.e., if $\{\bar u\neq 0\}$ is of
	measure zero which is the case in \ref{item:nonemptiness_Frechet_qs0_trivial}.
	Thus, let us assume that~\ref{item:nonemptiness_Frechet_qs0} holds.
	For each $k\in\N$, we define $\Omega_k:=\{\bar u\neq 0\}\cap\{\bar u+h_k=0\}$
	and obtain
	\begin{equation}\label{eq:some_estimate_nonemptiness_Frechet_qs0}
		\frac{q_{s,0}(\bar u+h_k)-q_{s,0}(\bar u)}{\norm{h_k}_s}
		\geq
		\frac{\int_{\{\bar u\neq 0\}}\bigl(|\bar u(x)+h_k(x)|_0-1\bigr)\,\mathrm dx}{\norm{h_k}_s}
		=
		-\frac{\blambda(\Omega_k)}{\norm{h_k}_s}.
	\end{equation}
	By $\norm{h_k}_s \searrow 0$, we get
	$\blambda(\Omega_k) \to 0$.
	In the case where $\blambda(\Omega_k) = 0$ holds along the tail of the sequence, we get 
	$\blambda(\Omega_k)/\norm{h_k}_s\to 0$.
	Otherwise, we may assume w.l.o.g.\ $\blambda(\Omega_k) > 0$ for all $k\in\N$.
	Thus,
	we have $\norm{h_k}_{s,\Omega_k} = \norm{\bar u}_{s,\Omega_k} > 0$ for all $k\in\N$
	and, consequently,
	\begin{equation*}
		-\frac{\blambda(\Omega_k)}{\norm{h_k}_{s}}
		\ge
		-\frac{\blambda(\Omega_k)}{\norm{h_k}_{s,\Omega_k}}
		=
		-\frac{\blambda(\Omega_k)}{\norm{\bar u}_{s,\Omega_k}}
		.
	\end{equation*}
	The latter term, however, tends to zero since $\bar u$ is an $s$-SD function.
	Thus, taking the limit inferior in \eqref{eq:some_estimate_nonemptiness_Frechet_qs0}
	yields $0\in\fsub q_{s,0}(\bar u)$ in the presence of ~\ref{item:nonemptiness_Frechet_qs0}.
	
	In order to show the converse statement, we assume that there exists some 
	$\eta\in\widehat{\partial}q_{s,0}(\bar u)$. \Cref{lem:trivial_upper_bound_Frechet_qs0}
	shows $\{\eta\neq 0\}\subset \{\bar u= 0\}$. 
	Suppose that $\bar u$ is not identically zero almost everywhere on $\Omega$.
	
	For $s=1$, choose $\rho>0$ such that $\Omega':=\{0<\abs{\bar u}\leq\rho\}$ is of positive measure.
	Next, we pick a sequence $\{\Omega_k'\}_{k\in\N}$ of measurable subsets of $\Omega'$
	which satisfy $\blambda(\Omega_k')\searrow 0$. 
	For each $k\in\N$, we set $h_k:=-\bar u\chi_{\Omega_k'}$.
	By construction, we have $\norm{h_k}_1\searrow 0$. Furthermore, we find
	\begin{align*}
		&
		\frac{q_{1,0}(\bar u+h_k)-q_{1,0}(\bar u)-\int_\Omega\eta(x)h_k(x)\,\mathrm dx}
			{\norm{h_k}_1}
		=
		-\frac{\blambda(\Omega_k')}{\int_{\Omega_k'}|\bar u(x)|\,\mathrm dx}
		\leq
		-\frac{\blambda(\Omega_k')}{\rho\,\blambda(\Omega_k')}
		=
		-
		\frac{1}{\rho},
	\end{align*}
	contradicting $\eta\in\widehat{\partial}q_{1,0}(\bar u)$.
	Consequently, $s>1$ holds.
	
	Finally, suppose that $\bar u$ is not an $s$-SD function.
	Then there is a sequence $\{\Omega_k\}_{k\in\N}$ of measurable subsets of
	$\{\bar u\neq 0\}$ such that $\blambda(\Omega_k)\searrow 0$ while the quotients
	$\blambda(\Omega_k)/\norm{\bar u}_{s,\Omega_k}$ do not converge
	to zero.
	For simplicity, we assume that there is $\beta>0$ such that 
	$\blambda(\Omega_k)/\norm{\bar u}_{s,\Omega_k}\geq\beta$
	holds for all $k\in\N$ (otherwise, consider a suitable subsequence).
	Once more, we make use of the sequence $\{h_k\}_{k\in\N}$ given by $h_k:=-\bar u\chi_{\Omega_k}$
	for each $k\in\N$. As above, we exploit $\{\eta\neq 0\}\subset \{\bar u= 0\}$ and
	$\{h_k\neq 0\}\subset\{\bar u\neq 0\}$ in order to find
	\begin{align*}
		\frac{q_{s,0}(\bar u+h_k)-q_{s,0}(\bar u)-\int_\Omega\eta(x)h_k(x)\,\mathrm dx}{\norm{h_k}_s}
		=
		-\frac{\blambda(\Omega_k)}{\norm{\bar u}_{s,\Omega_k}}
		\leq
		-\beta,
	\end{align*}
	yielding a contradiction to $\eta\in\widehat{\partial}q_{s,0}(\bar u)$ 
	since $\norm{h_k}_s\searrow 0$.
\end{proof}

Now, we are in position to fully characterize the Fr\'{e}chet subdifferential of $q_{s,0}$.
First, we investigate the case $s=1$ which needs to be treated separately.
\begin{theorem}\label{thm:characterization_Frechet_q10}
	We have 
	\[
		\forall\bar u\in L^1(\Omega)\colon\quad
		\widehat{\partial}q_{1,0}(\bar u)
		=
		\begin{cases}
			\{0\}		&	\text{if $\bar u=0$ a.e.\ on $\Omega$,}\\
			\varnothing	&	\text{otherwise.}
		\end{cases}
	\]
\end{theorem}
\begin{proof}
	Due to \cref{lem:nonemptiness_Frechet_qs0}, we already know that
	$\widehat{\partial}q_{1,0}(\bar u)$ is empty for each $\bar u\in L^1(\Omega)\setminus\{0\}$.
	Thus, assume that $\bar u$ vanishes almost everywhere on $\Omega$. 
	In the proof of \cref{lem:nonemptiness_Frechet_qs0}, we verified
	$0\in\widehat{\partial}q_{1,0}(\bar u)$. Consequently, we only need to show the converse inclusion.
	Thus, fix $\eta\in\widehat{\partial}q_{1,0}(\bar u)$ and assume that $\eta$ is not identically
	zero almost everywhere on $\Omega$. 
	Then we find a measurable set $\Omega'\subset\Omega$ of positive measure as well as
	some $\rho>0$ such that $|\eta(x)|\geq\rho$ holds for almost all $x\in\Omega'$.
	Consider a sequence $\{\Omega_k\}_{k\in\N}$ of measurable subsets of $\Omega'$ which satisfy
	$\blambda(\Omega_k)\searrow 0$. For each $k\in\N$, we define 
	$h_k:=\tfrac{2}{\rho}\chi_{\Omega_k}\sgn\eta$. 
	Clearly, $\norm{h_k}_1\searrow 0$ holds.
	Furthermore, we find
	\begin{align*}
		\frac{q_{1,0}(h_k)-\int_\Omega\eta(x)h_k(x)\,\mathrm dx}{\norm{h_k}_1}
		=
		\frac{\blambda(\Omega_k)-\tfrac{2}{\rho}\int_{\Omega_k}|\eta(x)|\,\mathrm dx}
			{\tfrac{2}{\rho}\blambda(\Omega_k)}
		\leq
		-\frac{\blambda(\Omega_k)}{\frac{2}{\rho}\blambda(\Omega_k)}=-\frac{\rho}{2}<0,
	\end{align*}
	contradicting $\eta\in\widehat{\partial}q_{1,0}(\bar u)$.
\end{proof}

\begin{remark}\label{rem:Frechet_q10_vs_PMP}
	Let us consider the unconstrained minimization of the function $f+q_{1,0}$ on $L^1(\Omega)$
	where $f\colon L^1(\Omega)\to\R$ is Fr\'{e}chet differentiable.
	Exploiting the sum rule from \cite[Proposition~1.107]{Mordukhovich2006} and 
	Fermat's rule from \cite[Proposition~1.114]{Mordukhovich2006}, a necessary condition
	for $\bar u\in L^1(\Omega)$ to be a local minimizer of $f+q_{1,0}$ is
	$-f'(\bar u)\in\widehat{\partial}q_{1,0}(\bar u)$.
	Due to \cref{thm:characterization_Frechet_q10}, this amounts to
	$\bar u= 0$ and $f'(\bar u)=0$ almost everywhere on $\Omega$.
	A similar result can be obtained when applying Pontryagin's maximum principle to
	the problem of interest, see
	\cite[Theorem~2.2]{ItoKunisch2014} or \cite[Section~2.1]{NatemeyerWachsmuth2020}.
\end{remark}

Next, we characterize the Fr\'{e}chet subdifferential of $q_{s,0}$ for $s\in(1,\infty)$.

\begin{theorem}\label{thm:characterization_Frechet_qs0}
	Fix $s \in (1,\infty)$.
	Then we have
	\begin{equation*}
	\forall \bar u\in L^s(\Omega)\colon\quad
	\fsub q_{s,0}(\bar u)
	=
	\begin{cases}
		\set{
			\eta \in L^r(\Omega)
			\given
			\{\eta\neq 0\} \subset \{\bar u= 0\}
		}
		&
		\text{if $\bar u$ is $s$-SD,}
		\\
		\varnothing & \text{otherwise}.
	\end{cases}
\end{equation*}
\end{theorem}
\begin{proof}
	Due to $s\in(1,\infty)$, $\fsub q_{s,0}(\bar u)$ is nonempty if and only if
	$\bar u\in L^s(\Omega)$ is an $s$-SD function, see \cref{lem:nonemptiness_Frechet_qs0}. 
	Thus, fix an $s$-SD function $\bar u\in L^s(\Omega)$.
	The inclusion ``$\subset$'' follows from \cref{lem:trivial_upper_bound_Frechet_qs0}.
	For the reverse inclusion,
	let $\eta \in L^r(\Omega)$ with $\{\eta\neq 0\} \subset \{\bar u= 0\}$ be given.
	We have to show 
	\begin{equation*}
		\liminf_{k \to \infty}
		\frac{
			q_{s,0}(\bar u + h_k) - q_{s,0}(\bar u) - \int_\Omega\eta(x)h_k(x)\,\mathrm dx
		}{
			\norm{h_k}_s
		}
		\ge 0
	\end{equation*}
	for all sequences $\seq{h_k}_{k \in \N} \subset L^s(\Omega)$
	with $\norm{h_k}_s \searrow 0$.
	For such a sequence, we set
	\begin{align*}
		D_k
		&:=
		\frac{
			q_{s,0}(\bar u + h_k) - q_{s,0}(\bar u) - \int_\Omega\eta(x)h_k(x)\,\mathrm dx
		}{
			\norm{h_k}_s
		}
		\\&
		=
		\frac{
			\int_{\{\bar u= 0\}} (\abs{h_k(x)}_0 - \eta(x) h_k(x)) \, \dx
		}{
			\norm{h_k}_s
		}
		+
		\frac{
			\int_{\{\bar u\neq 0\}} (\abs{\bar u(x) + h_k(x)}_0 - 1) \, \dx
		}{
			\norm{h_k}_s
		}
		=:
		D_k^1 + D_k^2.
	\end{align*}
	Let us validate $\liminf_{k \to \infty} D_k^1 \ge 0$.
	Using H\"older's inequality on $\{\bar u= 0\}\cap\{h_k\neq 0\}$
	and $\norm{h_k}_{s,\{\bar u= 0\}\cap\{h_k\neq 0\}}=\norm{h_k}_{s,\{\bar u= 0\}}$,
	we have
	\begin{align*}
		D_k^1 
		&\ge
		\frac{
			\int_{\{\bar u= 0\}}\abs{h_k(x)}_0\,\dx 
		}{
			\norm{h_k}_s
		}
		-
		\frac{ 
			\int_{\{\bar u=0\}}
			|\eta(x)h_k(x)|\,\dx
		}{
			\norm{h_k}_{s,\{\bar u= 0\}}
		}
		\\
		&
		\geq
		\frac{ \blambda(\{\bar u= 0\} \cap \{h_k\neq 0\}) }{ \norm{h_k}_{s} }
		-
		\norm{\eta}_{r, \{\bar u= 0\} \cap \{h_k\neq 0\}}
		.
	\end{align*}
	In case that 
	$\blambda(\{\bar u= 0\} \cap \{h_k\neq 0\}) \not\to 0$,
	this yields $D_k^1 \to \infty$.
	On the other hand,
	if we have
	$\blambda(\{\bar u= 0\} \cap \{h_k\neq 0\}) \to 0$,
	we get
	$\norm{\eta}_{r, \{\bar u= 0\} \cap \{h_k\neq 0\}} \to 0$.
	In any case, $\liminf_{k \to \infty} D_k^1 \ge 0$.

	It remains to check $\liminf_{k \to \infty} D_k^2 \ge 0$.
	This, however, can be distilled from the first part of the proof
	of \cref{lem:nonemptiness_Frechet_qs0} since $\bar u$ is an $s$-SD function.
	
	Combining these estimates, we have shown $\liminf_{k\to\infty}D_k\geq 0$
	which yields the claim.
\end{proof}

\begin{remark}\label{rem:Frechet_qs0_minimization}
	Similar to \cref{rem:Frechet_q10_vs_PMP}, we consider the unconstrained 
	minimization of the function $f+q_{s,0}$ on $L^s(\Omega)$ where $f\colon L^s(\Omega)\to\R$
	is Fr\'{e}chet differentiable and $s\in(1,\infty)$
	Then $-f'(\bar u)\in\fsub q_{s,0}(\bar u)$ is a necessary condition for $\bar u\in L^s(\Omega)$
	to be a local minimizer of $f+q_{s,0}$. 
	\Cref{thm:characterization_Frechet_qs0} now yields that $f'(\bar u)\in L^r(\Omega)$
	has to vanish on $\{\bar u\neq 0\}$. Moreover, the implicitly demanded 
	nonemptiness of $\fsub q_{s,0}(\bar u)$ requires that either $\bar u$ is equal to zero
	almost everywhere on $\Omega$ or that $\bar u$ tends to zero if at all slowly enough if 
	$\{\bar u\neq 0\}$
	is of positive measure since $\bar u$ must be an $s$-SD function, 
	see \cref{sec:s_SD_functions}. 
	In this regard, the obtained necessary optimality conditions clearly promote \emph{sparse} 
	controls $\bar u$.
\end{remark}

\subsection{Limiting subdifferential}\label{sec:limiting_p=0}

We now exploit \cref{thm:characterization_Frechet_qs0} in order to characterize the
limiting and singular subdifferential of $q_{s,0}$ for each $s\in(1,\infty)$. 
As already pointed out in \cref{sec:variational_analysis}, the limiting subdifferential of
$q_{1,0}$ coincides with its Fr\'{e}chet subdifferential due to \cite[Theorem~3.2]{Chieu2009}.
Anyway, the fact that $L^1(\Omega)$ is not an Asplund space underlines that the case $s=1$
might be of limited importance here.

\begin{theorem}\label{thm:characterization_limiting_qs0}
	Fix $s\in(1,\infty)$. Then we have
	\[
		\forall\bar u\in L^s(\Omega)\colon\quad
		\partial q_{s,0}(\bar u)=\{\eta\in L^r(\Omega)\,|\,\{\eta\neq 0\}\subset \{\bar u= 0\}\}.
	\]
\end{theorem}
\begin{proof}
	Fix $\bar u\in L^s(\Omega)$.
	In case where $\bar u= 0$ holds almost everywhere on $\Omega$, 
	\cref{thm:characterization_Frechet_qs0} already
	gives us $\fsub q_{s,0}(\bar u)=L^r(\Omega)$ which implies
	$\partial q_{s,0}(\bar u)=L^r(\Omega)$.
	Thus, we assume that $\{\bar u\neq 0\}$ possesses positive measure for the remainder of the proof 
	and verify both inclusions separately.
	
	In order to show the inclusion ``$\supset$'', we fix $\eta\in L^r(\Omega)$ satisfying
	$\{\eta\neq 0\}\subset \{\bar u= 0\}$.
	For each $k\in\N$, we define $\Omega_k:=\{\abs{\bar u}\geq 1/k\}$.
	Clearly, these sets are measurable and provide a nested exhaustion of $\{\bar u\neq 0\}$.
	Now, set $u_k:=\bar u\chi_{\Omega_k}$ for each $k\in\N$ and observe that
	$\{u_k=0\}\supset \{\bar u= 0\}$ holds.
	Invoking \cref{lem:slowly_decreasing_functions}, 
	$u_k$ is an $s$-SD function for each $k\in\N$, so that
	\cref{thm:characterization_Frechet_qs0} yields
	$\eta\in\widehat{\partial}q_{s,0}(u_k)$ for each $k\in\N$. Due to
	\begin{align*}
		\norm{u_k-\bar u}_s^s
		=
		\int_\Omega|\bar u(x)|^s(1-\chi_{\Omega_k}(x))\,\mathrm dx
		\leq
		\frac{\blambda(\Omega)}{k^s}
		\to 
		0,
	\end{align*}
	we find $u_k\to\bar u$ in $L^s(\Omega)$.
	Exploiting $\Omega_k\subset\{\bar u\neq 0\}$ for each $k\in\N$ and lower
	semicontinuity of $q_{s,0}$, see \cref{lem:lower_semicontinuity_of_qs0}, we find
	\begin{align*}
		\blambda(\{\bar u\neq 0\})
		=
		q_{s,0}(\bar u)
		&
		\leq
		\liminf\limits_{k\to\infty}q_{s,0}(u_k)
		\\
		&
		\leq
		\limsup\limits_{k\to\infty}q_{s,0}(u_k)
		=
		\limsup\limits_{k\to\infty}\blambda(\Omega_k)
		\leq
		\blambda(\{\bar u\neq 0\}),
	\end{align*}
	i.e., $q_{s,0}(u_k)\to q_{s,0}(\bar u)$.
	Thus, by definition of the limiting subdifferential, 
	we have shown $\eta\in \partial q_{s,0}(\bar u)$.
	
	In order to prove ``$\subset$'', we fix $\eta\in\partial q_{s,0}(\bar u)$.
	Thus, we find sequences $\{u_k\}_{k\in\N}\subset L^s(\Omega)$ and
	$\{\eta_k\}_{k\in\N}\subset L^r(\Omega)$ which satisfy 
	$u_k\to\bar u$ in $L^s(\Omega)$,
	$q_{s,0}(u_k)\to q_{s,0}(\bar u)$, 
	$\eta_k\weakly\eta$ in $L^r(\Omega)$, and
	$\eta_k\in\fsub q_{s,0}(u_k)$ for all $k\in\N$. 
	Along a subsequence (without relabeling), we may assume that $\{u_k\}_{k\in\N}$ converges pointwise
	almost everywhere to $\bar u$.
	Thus, for almost every $x\in\{\bar u\neq 0\}$, we have $u_k(x)\to\bar u(x)\neq 0$, i.e.,
	$x\in\{u_k\neq 0\}$ and, thus, $x\in \{\eta_k=0\}$ for sufficiently large $k\in\N$. 
	Thus, almost everywhere on $\{\bar u\neq 0\}$, $\{\eta_k\}_{k\in\N}$ converges pointwise to $0$.
	From $\eta_k\weakly\eta$ in $L^r(\Omega)$, we infer that the weak limit needs to
	vanish on $\{\bar u\neq 0\}$, i.e., $\{\bar u\neq 0\}\subset \{\eta= 0\}$.
	This, however, also means $\{\eta\neq 0\}\subset \{\bar u= 0\}$.
\end{proof}

Reprising the above proof while incorporating some nearby minor adjustments, one can show the following
result regarding the singular subdifferential of $q_{s,0}$.
\begin{theorem}\label{thm:characterization_singular_qs0}
	Fix $s\in(1,\infty)$. Then we have
	\[
		\forall\bar u\in L^s(\Omega)\colon\quad
		\partial^\infty q_{s,0}(\bar u)=\{\eta\in L^r(\Omega)\,|\,\{\eta\neq 0\}\subset \{\bar u= 0\}\}.
	\]
\end{theorem}

As a corollary of \cref{thm:characterization_Frechet_qs0,thm:characterization_singular_qs0}, 
we can fully characterize the Lipschitzian properties of $q_{s,0}$.

\begin{corollary}\label{cor:Lipschitzness_qs0}
	For $s\in(1,\infty)$, $q_{s,0}$ is nowhere Lipschitz continuous. 
\end{corollary}
\begin{proof}
	Using \cref{lem:Lipschitzness_s>1}, \cref{thm:characterization_singular_qs0} shows that
	$q_{s,0}$ cannot be Lipschitz continuous at all points $\bar u\in L^s(\Omega)$
	which satisfy $\blambda(\{\bar u=0\})>0$ since $\partial^\infty q_{s,0}(\bar u)$ does not reduce
	to $\{0\}$ in this situation. 
	
	Thus, let us consider $\bar u\in L^s(\Omega)$ such that $\bar u\neq 0$ holds almost everywhere
	on $\Omega$. 
	In the reminder of this proof, we show that $q_{s,0}$ violates the condition from
	\cref{lem:Lipschitzness_s>1}~\ref{item:SNEC} at $\bar u$ which implies that $q_{s,p}$
	cannot be Lipschitz at $\bar u$.
	Thus, pick a scalar $\alpha>0$ such that $\{\abs{\bar u}\geq\alpha\}$ possesses positive
	measure and set $\Omega_k:=\{\abs{\bar u}\geq\alpha/k\}$ for each $k\in\N$.
	By construction, $\{\Omega_k\}_{k\in\N}$ is an exhaustion of $\{\bar u\neq 0\}$, and each
	of the sets $\Omega_k$, $k\in\N$, possesses positive measure.
	Thus, we can pick a sequence $\{\Omega_k'\}_{k\in\N}$ of measurable subsets of
	$\Omega$ such that $\blambda(\Omega_k')\searrow 0$ and, for each $k\in\N$, 
	$\Omega_k'\subset\Omega_k$.
	For each $k\in\N$, we define $u_k:=\bar u\chi_{\Omega_k\setminus\Omega_k'}$ and 
	$\eta_k:=\blambda(\Omega_k')^{-1/r}\chi_{\Omega_k'}$.
	Similar as in the proof of \cref{thm:characterization_limiting_qs0}, 
	we can show $u_k\to\bar u$ in $L^s(\Omega)$ and $q_{s,0}(u_k)\to q_{s,0}(\bar u)$.
	Furthermore, for each $h\in L^s(\Omega)$, we find
	\begin{align*}
		\left|\int_\Omega \eta_k(x)h(x)\,\dx\right|
		&\leq
		\norm{h}_{s,\Omega_k'}\norm{\eta_k}_{r,\Omega_k'}
		=
		\norm{h}_{s,\Omega_k'}
	\end{align*}
	by applying H\"older's inequality on $\Omega_k'$, 
	and due to $\norm{h}_{s,\Omega_k'}\to 0$, the above estimate
	yields $\eta_k\weakly 0$ in $L^r(\Omega)$. Furthermore, $\norm{\eta_k}_r=1$ for each
	$k\in\N$ guarantees that this convergence is not strong.
	Finally, observe that due to \cref{lem:slowly_decreasing_functions} and
	\cref{thm:characterization_Frechet_qs0},
	we find $\eta_k\in\tfrac1k\fsub q_{s,0}(u_k)$ for each $k\in\N$.
	Thus, \cref{lem:Lipschitzness_s>1} shows that $q_{s,0}$ cannot be Lipschitz continuous at $\bar u$.
\end{proof}

\section{The case $p\in(0,1)$}\label{sec:0<p<1}

Throughout the section, we assume that $p\in(0,1)$ holds.
Here, we study the variational properties of the functional $q_{s,p}$. 
Basically, although some proofs seem to be a little technical,
we proceed in similar way as in \cref{sec:p=0} in order to
compute the subdifferentials of interest. 

\subsection{Fr\'{e}chet subdifferential}

Again, we start to prove validity of a natural upper bound for the Fr\'{e}chet subdifferential
of $q_{s,p}$.
\begin{lemma}
	\label{lem:subderivative_by_derivative}
	For given $s\in[1,\infty)$ and $\bar u\in L^s(\Omega)$, we have
	\[
		\widehat{\partial}q_{s,p}(\bar u)
		\subset
		\{
			\eta\in L^r(\Omega)
			\,|\,
			\eta = p \abs{\bar u}^{p-2} \bar u
			\text{ a.e.\ on } \{\bar u\neq 0\}
		\}.
	\]
\end{lemma}
\begin{proof}
	Let $\eta  \in \widehat{\partial}q_{s,0}(\bar u)$
	be given.
	For $\varepsilon > 0$, we set $A_\varepsilon := \{\abs{\bar u} > \varepsilon\}$.
	For an arbitrary measurable subset $B \subset A_\varepsilon$ of positive measure, we
	define a sequence $\{h_k\}_{k\in\N}\subset L^s(\Omega)$ by means of
	$h_k := k^{-1} \chi_{B}$ for each $k\in\N$.
	Clearly, we have $\norm{h_k}_s\searrow 0$, so the definition of the Fr\'{e}chet subdifferential
	yields
	\begin{align*}
		0
		&\le
		\blambda(B)^{1/s}
		\liminf_{k \to \infty}
		\frac{
			q_{s,p}(\bar u + h_k) - q_{s,p}(\bar u) - \int_\Omega \eta(x)h_k(x)\,\mathrm dx
		}{
			\norm{h_k}_s
		}
		\\&=
		\liminf_{k \to \infty}
		\int_{B} \bigl(k \parens*{ \abs{\bar u(x) + 1/k}^p - \abs{\bar u(x)}^p } - \eta(x)\bigr) \, \dx
		=
		\int_{B} \bigl(p \abs{\bar u(x)}^{p-2} \bar u(x) - \eta(x)\bigr) \, \dx
		.
	\end{align*}
	Note that we used
	the dominated convergence theorem with the integrable, dominating function
	$(p\varepsilon^{p-1}+|\eta|)\chi_B$
	for the last equality.
	Similarly, we can use the sequence $\{\tilde h_k\}_{k\in\N}\subset L^s(\Omega)$
	given by $\tilde h_k := -k^{-1}\chi_B$ for each $k\in\N$
	to obtain the reverse inequality.
	Since $B \subset A_\varepsilon$ was arbitrary,
	this shows
	$\eta = p \abs{\bar u}^{p-2} \bar u$ almost everywhere on $A_\varepsilon$.
	Since $\{\bar u\neq 0\} = \bigcup_{\varepsilon>0} A_{\varepsilon}$ holds,
	the claim has been shown.
\end{proof}

We note that, technically, the above proof also applies to the setting $p = 0$
and, thus, provides another possible validation of \cref{lem:trivial_upper_bound_Frechet_qs0}.
However, let us emphasize that the proof we provided for \cref{lem:trivial_upper_bound_Frechet_qs0}
is much simpler and does not exploit deeper results from integration theory like the dominated
convergence theorem.

Similar to \cref{lem:nonemptiness_Frechet_qs0}, we aim to characterize all points in
$L^s(\Omega)$ where the associated Fr\'{e}chet subdifferential of $q_{s,p}$ is nonempty.
\begin{lemma}\label{lem:nonemptiness_Frechet_qsp}
	Fix $s\in[1,\infty)$ and $\bar u\in L^s(\Omega)$.
	Then $\fsub q_{s,p}(\bar u)$ is nonempty if and only if one of the following conditions is valid:
	\begin{enumerate}
		\item\label{item:nonemptiness_Frechet_qsp_trivial} 
			$\bar u=0$ holds almost everywhere on $\Omega$,
		\item\label{item:nonemptiness_Frechet_qsp} 
			it holds $s>1$ and 
			$\abs{\bar u}^{p-1}\chi_{\{\bar u\neq 0\}}\in L^r(\Omega)$ is valid.
	\end{enumerate}
\end{lemma}
\begin{proof}
	In the first part of this proof, we show that, in the presence
	of~\ref{item:nonemptiness_Frechet_qsp_trivial} or~\ref{item:nonemptiness_Frechet_qsp},
	 $\bar\eta\colon\Omega\to\R$ given by $\bar\eta:=p\abs{\bar u}^{p-2}\bar u\chi_{\{\bar u\neq 0\}}$
	belongs to $\fsub q_{s,p}(\bar u)$. This is clearly obvious in case where $\bar u$ 
	vanishes almost everywhere on $\Omega$,
	i.e., when~\ref{item:nonemptiness_Frechet_qsp_trivial} holds,
	so let us focus on the situation given in~\ref{item:nonemptiness_Frechet_qsp}. 
	First, we observe that $\bar\eta$ defined above is a function from $L^r(\Omega)$ due to the
	requirements in~\ref{item:nonemptiness_Frechet_qsp}. Next, we show that for each
	sequence $\{h_k\}_{k\in\N}\subset L^s(\Omega)$ satisfying $\norm{h_k}_s\searrow 0$, we have
	\begin{equation}\label{eq:special_eta_in_fsub}
		\liminf\limits_{k\to\infty}
		\frac{
			\int_{\{\bar u\neq 0\}}
					\bigl(|\bar u(x)+h_k(x)|^p-|\bar u(x)|^p-p|\bar u(x)|^{p-2}\bar u(x)h_k(x)\bigr)
					\,\dx
		}{
		\norm{h_k}_s
		}
		\geq 0.
	\end{equation}
	One can easily check that by definition of the Fr\'{e}chet subdifferential and $\bar\eta$, 
	this is sufficient for $\bar\eta\in\fsub q_{s,p}(\bar u)$.
	
	It will be beneficial to write $h_k =c_k\bar u+h_k\chi_{\{\bar u=0\}}$ for each $k\in\N$ where
	the measurable function $c_k\colon\Omega\to\R$ is given by
	$c_k := h_k \bar u^{-1} \chi_{\set{\bar u \ne 0}}$.
	Furthermore, we will make use of the set 
	$\Omega_k:=\{h_k\neq 0\}\cap\{\bar u\neq 0\}$ for each $k\in\N$.
	With the aid of these definitions, we can rewrite the quotient of interest by means of
	\begin{equation}\label{eq:quotient_of_interest}
		\frac{
			\int_{\Omega_k}
				\bigl(|1+c_k(x)|^p-1-pc_k(x)\bigr)|\bar u(x)|^p\,\dx
		}{
			\left(\int_{\Omega_k}|c_k(x)|^s|\bar u(x)|^s\,\dx
				+\int_{\{\bar u=0\}}|h_k(x)|^s\,\dx\right)^{1/s}
		}.
	\end{equation}
	Next, for each $k\in\N$, we decompose $\Omega_k$ into the four disjoint subsets
	\begin{align*}
		\Omega_k^1&:=\{c_k<-1/p\},&
		\Omega_k^2&:=\{-1/p\leq c_k\leq -1/2\},&\\
		\Omega_k^3&:=\{-1/2<c_k<1/2\},&
		\Omega_k^4&:=\{c_k\geq 1/2\}.&
	\end{align*}
	This allows us to rewrite the quotient in \eqref{eq:quotient_of_interest} 
	as $Q_k^1+Q_k^2+Q_k^3+Q_k^4$
	with
	\[
		Q_k^i
		:=
		\frac{
			\int_{\Omega_k^i}
				\bigl(|1+c_k(x)|^p-1-pc_k(x)\bigr)|\bar u(x)|^p\,\dx
		}{
			\left(\int_{\Omega_k}|c_k(x)|^s|\bar u(x)|^s\,\dx
				+\int_{\{\bar u=0\}}|h_k(x)|^s\,\dx\right)^{1/s}
		}
	\]
	for each $i=1,2,3,4$. By construction, $Q_k^1$ is nonnegative which yields
	$\liminf_{k\to\infty}Q_k^1\geq 0$. Furthermore, in case where $\Omega_k^2$ is of
	positive measure, we find
	\begin{align*}
		Q_k^2
		&\geq
		(p-2)\frac{\norm{\bar u}_{p,\Omega_k^2}^p}{\norm{\bar u}_{s,\Omega_k^2}}
		\geq
		(p-2)
		\frac{\norm{\bar u}_{s,\Omega_k^2} \norm{\abs{\bar u}^{p-1}}_{r,\Omega_k^2}
		}{
		\norm{\bar u}_{s,\Omega_k^2}}
		=
		(p-2)
		\norm{\abs{\bar u}^{p-1}}_{r,\Omega_k^2}
	\end{align*}
	from $\abs{\bar u}^{p-1}\chi_{\{\bar u\neq 0\}}\in L^r(\Omega)$ and H\"older's 
	inequality on $\Omega_k^2$.
	Since we have $\blambda(\Omega_k^2)\searrow 0$ from $\norm{h_k}_s\searrow 0$,
	$\norm{\abs{\bar u}^{p-1}}_{r,\Omega_k^2}\searrow 0$ holds which is why
	$\liminf_{k\to\infty}Q^2_k\geq 0$ follows.
	Next, let us investigate the setting where $\Omega_k^3$ is of positive measure.
	A second-order Taylor expansion of the mapping $y\mapsto (1+y)^p-1$ at the origin
	yields $(1+y)^p-1-py\geq -y^2$ for all $y\in(-1/2,1/2)$. 
	Thus, we obtain
	\begin{align*}
		Q_k^3
		\geq
		-\frac{\int_{\Omega^3_k}|c_k(x)|^2|\bar u(x)|^p\,\dx}{\norm{c_k\bar u}_{s,\Omega_k^3}}
		\geq
		-\frac{\norm{c_k\bar u}_{s,\Omega_k^3}\norm{c_k\abs{\bar u}^{p-1}}_{r,\Omega_k^3}}
		{\norm{c_k\bar u}_{s,\Omega_k^3}}
		=
		-
		\norm{c_k\abs{\bar u}^{p-1}}_{r,\Omega_k^3}
	\end{align*}	
	where we used H\"older's inequality on $\Omega_k^3$ and 
	$\abs{\bar u}^{p-1}\chi_{\{\bar u\neq 0\}}\in L^r(\Omega)$ which, by boundedness of
	$c_k$ on $\Omega_k^3$, guarantees $c_k\abs{\bar u}^{p-1}\chi_{\Omega_k^3}\in L^r(\Omega)$.
	Observing that $\bar u$ does not vanish on $\Omega_k^3$, that $h_k=c_k\bar u$ holds on
	$\Omega_k^3$, and that $\norm{h_k}_s\searrow 0$ is valid, we obtain the 
	pointwise convergence of $\{c_k\}_{k\in\N}$ to $0$ almost everywhere on $\Omega_k^3$. Thus,
	$c_k(x)|\bar u(x)|^{p-1}\to 0$ holds for almost every $x\in\Omega_k^3$.
	Noting that $\{c_k\abs{\bar u}^{p-1}\chi_{\Omega_k^3}\}_{k\in\N}$ is dominated by 
	$\abs{\bar u}^{p-1}\chi_{\Omega_k^3}\in L^r(\Omega)$,
	we find $\norm{c_k\abs{\bar u}^{p-1}}_{r,\Omega_k^3}\to 0$ from Lebesgue's dominated 
	convergence theorem, i.e., $\liminf_{k\to\infty}Q_k^3\geq 0$.
	Finally, we address the situation where $\Omega_k^4$ is of positive measure.
	Recalling that $h_k=c_k\bar u$,  $\bar u\neq 0$, and $c_k\geq 1/2$ hold on $\Omega_k^4$, 
	$\norm{h_k}_s\searrow 0$ implies $\blambda(\Omega_k^4)\searrow 0$.
	Exploiting $\abs{\bar u}^{p-1}\chi_{\Omega_k^4}\in L^r(\Omega)$, we find
	\begin{align*}
		Q_k^4
		&\geq 
		-p
		\frac{
			\int_{\Omega_k^4}c_k(x)|\bar u(x)|^p\,\dx
		}{
			\norm{c_k\bar u}_{s,\Omega_k^4}
		}
		\geq
		-p
		\frac{
			\norm{c_k\bar u}_{s,\Omega_k^4}\norm{\abs{\bar u}^{p-1}}_{r,\Omega_k^4}
		}{
			\norm{c_k\bar u}_{s,\Omega_k^4}
		}
		=
		-p\norm{\abs{\bar u}^{p-1}}_{r,\Omega_k^4}
	\end{align*}
	by applying H\"older's inequality on $\Omega_k^4$. Due to
	$\blambda(\Omega_k^4)\searrow 0$, we have $\norm{\abs{\bar u}^{p-1}}_{r,\Omega_k^4}\searrow 0$
	which yields $\liminf_{k\to\infty}Q_k^4\geq 0$.
	Combining all these estimates, \eqref{eq:special_eta_in_fsub} has been shown, i.e.,
	$\bar\eta\in\fsub q_{s,p}(\bar u)$ is valid.

	Let us show the converse statement. 
	Therefore, we assume that there is some $\eta\in\fsub q_{s,p}(\bar u)$.
	Due to \cref{lem:subderivative_by_derivative}, we know that $\eta=p \abs{\bar u}^{p-2}\bar u$
	holds almost everywhere on $\{\bar u\neq 0\}$.
	Thus, from $\eta\in L^r(\Omega)$, the condition 
	$\abs{\bar u}^{p-1}\chi_{\{\bar u\neq 0\}}\in L^r(\Omega)$ follows.
	We assume that $\bar u$ is not identically zero almost everywhere on $\Omega$.
	
	Suppose that $s=1$ holds. 
	We fix a set $\Omega'\subset\{\bar u\neq 0\}$ of positive measure and some
	$\rho>0$ such that $|\bar u(x)|\geq\rho$ holds almost everywhere on $\Omega'$.
	Set $A_\ell:=\{x\in\Omega'\,|\,0<|\bar u(x)|\leq \ell\}$ for each $\ell\in\N$.
	Then we have $\bigcup_{\ell\in\N} A_\ell=\Omega'$,
	and due to $\blambda(\Omega')>0$,
	there exists some $\ell_0\in\N$ such that $A_{\ell_0}$ is of positive measure.
	Let us now fix a sequence $\{\Omega_k'\}_{k\in\N}$ of measurable subsets of $A_{\ell_0}$ 
	which satisfy $\blambda(\Omega'_k)\searrow 0$.
	For brevity of notation, set $m_k:=\blambda(\Omega_k')$ for each $k\in\N$ and define
	$h_k:=m_k^{-1/2}\bar u\chi_{\Omega'_k}$.
	By construction, we have
	\[
		\norm{h_k}_1
		=
		m_k^{-1/2}\int_{\Omega'_k}|\bar u(x)|\,\mathrm dx
		\leq
		\ell_0\,m_k^{1/2}\searrow 0.
	\]
	Furthermore, we find
	\begin{equation}\label{eq:Frechet_q1p_empty_estimate}
		\begin{aligned}
		&\frac{q_{1,p}(\bar u+h_k)-q_{1,p}(\bar u)-\int_\Omega\eta(x)h_k(x)\,\mathrm dx}
			{\norm{h_k}_1}
		\\
		&\qquad
		=
		\frac{\int_{\Omega'_k}
			\bigl(\bigl(1+m_k^{-1/2}\bigr)^p|\bar u(x)|^p-|\bar u(x)|^p
				-pm_k^{-1/2}|\bar u(x)|^p\bigr)\,\mathrm dx
				}
			{m_k^{-1/2}\int_{\Omega'_k}|\bar u(x)|\,\mathrm dx}\\
		&\qquad
		=
		\frac{\bigl(1+m_k^{-1/2}\bigr)^p-1-pm_k^{-1/2}}{m_k^{-1/2}}
		\,
		\frac{\norm{\bar u}_{p,\Omega_k'}^p}{\norm{\bar u}_{1,\Omega_k'}}.
		\end{aligned}
	\end{equation}
	Due to $p\in(0,1)$, it holds
	\begin{align*}
		\frac{\bigl(1+m_k^{-1/2}\bigr)^p-1-pm_k^{-1/2}}{m_k^{-1/2}}
		=
		\left(m_k^{1/(2p)}+m_k^{1/(2p)-1/2}\right)^p
		-m_k^{1/2}-p
		\to 
		-p.
	\end{align*}
	On the other hand, we have
	\begin{align*}
		\frac{\norm{\bar u}_{p,\Omega_k'}^p}{\norm{\bar u}_{1,\Omega_k'}}
		\geq
		\frac{\rho^p}{\ell_0}
	\end{align*}
	by choice of $\Omega'_k\subset A_{\ell_0}$.
	Hence, for sufficiently large $k\in\N$, \eqref{eq:Frechet_q1p_empty_estimate} yields
	\begin{align*}
		\frac{q_{1,p}(\bar u+h_k)-q_{1,p}(\bar u)-\int_\Omega\eta(x)h_k(x)\,\mathrm dx}
			{\norm{h_k}_1}
		\leq
		-\frac{p}{2}\,\frac{\rho^p}{\ell_0}<0,
	\end{align*}
	but this contradicts $\eta\in\fsub q_{1,p}(\bar u)$.
\end{proof}

For $s>1$, fix a function $\bar u\in L^s(\Omega)$ such that $\{\bar u\neq 0\}$ possesses positive
measure and some sequence $\{\Omega_k\}_{k\in\N}$ of measurable subsets of $\{\bar u\neq 0\}$
possessing positive measure. Supposing that
$\abs{\bar u}^{p-1}\chi_{\{\bar u\neq 0\}}\in L^r(\Omega)$ is valid, 
H\"older's inequality on $\Omega_k$ yields
\begin{align*}
	\norm{\bar u}_{p,\Omega_k}^p
	\leq
	\norm{\bar u}_{s,\Omega_k}\norm{\abs{\bar u}^{p-1}}_{r,\Omega_k}
\end{align*}
for each $k\in\N$. Thus, we find
\begin{equation}\label{eq:s_SD_p>0}
	\blambda(\Omega_k)\searrow 0
	\qquad
	\Longrightarrow
	\qquad
	\frac{\norm{\bar u}_{p,\Omega_k}^p}{\norm{\bar u}_{s,\Omega_k}}
	\searrow 0,
\end{equation}
which can be interpreted as a reasonable adaptation of the $s$-SD property 
from \cref{def:slowly_decrasing_functions} to the setting
$p\in(0,1)$. Note that \eqref{eq:s_SD_p>0} can be used in the proof
of \cref{lem:nonemptiness_Frechet_qsp} in order to show $\liminf_{k\to\infty}Q_k^2\geq 0$.
However, as demonstrated above, \eqref{eq:s_SD_p>0} is implied by 
$\abs{\bar u}^{p-1}\chi_{\{\bar u\neq 0\}}\in L^r(\Omega)$ which, either way, needs to be
postulated in order to show the assertion of \cref{lem:nonemptiness_Frechet_qsp}. 
We can interpret $L^r$-regularity of $\abs{\bar u}^{p-1}\chi_{\{\bar u\neq 0\}}$ 
again as a condition which ensures that whenever $\bar u$ approaches
zero on $\{\bar u\neq 0\}$, then this has to happen slowly enough.
Recall that in case $p=0$, see \cref{lem:nonemptiness_Frechet_qs0}, 
$\abs{\bar u}^{-1}\chi_{\{\bar u\neq 0\}}\in L^r(\Omega)$ is only sufficient but not necessary
for the nonemptiness of $\fsub q_{s,0}(\bar u)$, see
\cref{lem:SD_by_hoelder} and \cref{ex:SD_not_Lr} as well.

Now, we are in position to fully characterize the Fr\'{e}chet subdifferential of $q_{s,p}$.
Again, we distinguish the cases $s=1$ and $s\in(1,\infty)$.

\begin{theorem}\label{thm:characterization_Frechet_q1p}
	We have
	\[
		\forall \bar u\in L^1(\Omega)\colon\quad
		\fsub q_{1,p}(\bar u)
		=
		\begin{cases}
			\{0\}	&	\text{if $\bar u=0$ a.e.\ on $\Omega$,}\\
			\varnothing		&	\text{otherwise}.
		\end{cases}
	\]
\end{theorem}
\begin{proof}
	Due to \cref{lem:nonemptiness_Frechet_qsp}, we know that $\fsub q_{1,p}(\bar u)$
	is empty for each $\bar u\in L^1(\Omega)\setminus\{0\}$.
	Thus, let us assume that $\bar u=0$ holds almost everywhere on $\Omega$.
	It is obvious by definition of the Fr\'{e}chet subdifferential that
	$0\in\fsub q_{1,p}(\bar u)$ is valid.
	In order to show the converse inclusion, fix $\eta\in\fsub q_{1,p}(\bar u)$
	and assume that $\eta$ does not vanish almost everywhere on $\Omega$.
	Then we find a measurable set $\Omega'\subset\Omega$ of positive measure 
	as well as some $\rho>0$ such that $|\eta(x)|\geq\rho$ holds for
	almost every $x\in\Omega'$. We fix a sequence $\{\Omega_k\}_{k\in\N}$ of measurable
	subsets of $\Omega'$ such that $\blambda(\Omega_k)\searrow 0$ is valid.
	Furthermore, we choose some constant $\alpha>\rho^{1/(p-1)}$. Due to $p\in(0,1)$,
	this yields $\alpha^{p-1}<\rho$.
	Now, we set $h_k:=\alpha\chi_{\Omega_k}\sgn\eta$ for each $k\in\N$ and observe that
	$\norm{h_k}_1\searrow 0$ is valid. Additionally, we find
	\begin{align*}
		\frac{q_{1,p}(h_k)-\int_\Omega\eta(x)h_k(x)\,\mathrm dx}{\norm{h_k}_1}
		&=
		\frac{\alpha^p\blambda(\Omega_k)-\alpha\int_{\Omega_k}|\eta(x)|\,\mathrm dx}
			{\alpha\blambda(\Omega_k)}\\
		&
		\leq
		\frac{\alpha^{p-1}\blambda(\Omega_k)-\rho\blambda(\Omega_k)}{\blambda(\Omega_k)}
		=
		\alpha^{p-1}-\rho
		<
		0,
	\end{align*}
	contradicting our assumption $\eta\in\fsub q_{1,p}(\bar u)$.
\end{proof}

\begin{theorem}\label{thm:characterization_Frechet_qsp}
	Fix $s\in(1,\infty)$.
	Then we have
	\[
		\forall \bar u\in L^s(\Omega)\colon\quad
		\fsub q_{s,p}(\bar u)
		=
			\{
				\eta\in L^r(\Omega)\,|\,
				\eta = p \abs{\bar u}^{p-2} \bar u
				\text{ a.e.\ on } \{\bar u\neq 0\}
			\}.
	\]
	In particular, the set on the right-hand side is empty if
	$\abs{\bar u}^{p-1}\chi_{\{\bar u\neq 0\}} \not\in L^r(\Omega)$.
\end{theorem}
\begin{proof}
	The inclusion ``$\subset$'' follows from \cref{lem:subderivative_by_derivative}.
	For the reverse inclusion,
	let $\eta \in L^r(\Omega)$ with $\eta=p\abs{\bar u}^{p-2}\bar u$ almost everywhere on
	$\{\bar u\neq 0\}$ be given.
	We have to show 
	\begin{equation*}
		\liminf_{k \to \infty}
		\frac{
			q_{s,p}(\bar u + h_k) - q_{s,p}(\bar u) - \int_\Omega\eta(x)h_k(x)\,\mathrm dx
		}{
			\norm{h_k}_s
		}
		\ge 0
	\end{equation*}
	for all sequences $\seq{h_k}_{k \in \N} \subset L^s(\Omega)$
	with $\norm{h_k}_s \searrow 0$.
	For such a sequence, we set
	\begin{align*}
		D_k
		&:=
		\frac{
			q_{s,p}(\bar u + h_k) - q_{s,p}(\bar u) - \int_\Omega\eta(x)h_k(x)\,\mathrm dx
		}{
			\norm{h_k}_s
		}
		\\&
		=
		\frac{
			\int_{\{\bar u= 0\}} \bigl(\abs{h_k(x)}^p - \eta(x) h_k(x)\bigr) \, \dx
		}{
			\norm{h_k}_s
		}\\
		&\qquad
		+
		\frac{
			\int_{\{\bar u\neq 0\}} 
				\bigl(\abs{\bar u(x) + h_k(x)}^p 
				- \abs{\bar u(x)}^p-p\abs{\bar u(x)}^{p-2}\bar u(x)h_k(x)
				\bigr) \, \dx
		}{
			\norm{h_k}_s
		}
		\\&
		=:
		D_k^1 + D_k^2.
	\end{align*}
	First, let us validate $\liminf_{k \to \infty} D_k^1 \ge 0$.
	In case where $\eta$ equals zero almost everywhere on $\{\bar u=0\}$, this is obvious.
	Otherwise, for some arbitrarily chosen $\varepsilon>0$, choose $t>0$ large enough such that
	$\norm{\eta\chi_{\{|\eta|>t\}}}_{r,\{\bar u=0\}}\leq\varepsilon$.
	Next, for each $k\in\N$, we define $h_k^1,h_k^2\in L^s(\Omega)$ by means of
	$h_k^1:=h_k\chi_{\{|h_k|\leq t^{1/(p-1)}\}}$ and $h_k^2:=h_k\chi_{\{|h_k|>t^{1/(p-1)}\}}$.
	By construction, we find
	\begin{equation}\label{eq:decomposition_for_Dk1}
		\begin{aligned}
		D_k^1
		=
		\frac{
			\int_{\{\bar u=0\}}\bigl(|h_k^1(x)|^p - \eta(x) h_k^1(x)\bigr) \, \dx
		}{
			\norm{h_k}_s
		}
		+
		\frac{
			\int_{\{\bar u=0\}}\bigl(|h_k^2(x)|^p - \eta(x) h_k^2(x)\bigr) \, \dx
		}{
			\norm{h_k}_s
		}.
		\end{aligned}
	\end{equation}
	Observing that for all $x\in\{\bar u=0\}\cap\{|h_k|\leq t^{1/(p-1)}\}\cap\{|\eta|\leq t\}$,
	we have the estimate $|h_k(x)|^{p-1}\geq t\geq|\eta(x)|$, i.e., 
	$|h_k(x)|^p\geq|\eta(x)h_k(x)|\geq\eta(x)h_k(x)$, it holds
	\begin{align*}
		\frac{
			\int_{\{\bar u=0\}}\bigl(|h_k^1(x)|^p - \eta(x) h_k^1(x)\bigr) \, \dx
		}{
			\norm{h_k}_s
		}
		&\geq 
		-
		\frac{
			\int_{\{\bar u=0\}}|\eta(x)\chi_{\{|\eta|>t\}}(x)h_k^1(x)|\,\dx
		}{
			\norm{h_k^1}_{s,\{\bar u=0\}}
		}\\
		&
		\geq
		-\norm{\eta\chi_{\{|\eta|>t\}}}_{r,\{\bar u=0\}}
		\geq 
		-\varepsilon
	\end{align*}
	where H\"older's inequality on $\{\bar u=0\}$ was used to obtain the last but one estimate.
	On the other hand, we find
	\begin{align*}
		\frac{
			\int_{\{\bar u=0\}}\bigl(|h_k^2(x)|^p - \eta(x) h_k^2(x)\bigr) \, \dx
		}{
			\norm{h_k}_s
		}
		&\geq
		-
		\frac{
			\int_{\{\bar u=0\}}|\eta(x)h_k^2(x)|\,\mathrm dx
		}{
			\norm{h_k^2}_{s,\{\bar u=0\}}
		}
		\geq
		-
		\norm{\eta\chi_{\{|h_k|>t^{1/(p-1)}\}}}_{r,\{\bar u=0\}}
	\end{align*}
	again from H\"older's inequality on $\{\bar u=0\}$.
	As a consequence, \eqref{eq:decomposition_for_Dk1} yields the estimate
	$D_k^1\geq -\varepsilon-\norm{\eta\chi_{\{|h_k|>t^{1/(p-1)}\}}}_{r,\{\bar u=0\}}$
	for all $k\in\N$.
	Since we have $\norm{h_k}_s\searrow 0$, the convergence
	$\blambda(\{|h_k|>t^{1/(p-1)}\})\to 0$ holds which guarantees
	$\norm{\eta\chi_{\{|h_k|>t^{1/(p-1)}\}}}_{r,\{\bar u=0\}}\to 0$.
	Observing that $\varepsilon>0$ is independent of $k$ and can be made arbitrarily small, 
	we have shown
	$\liminf_{k\to\infty}D_k^1\geq 0$.
	
	Noting that we can distill $\liminf_{k\to\infty}D_k^2\geq 0$ from the proof of
	\cref{lem:nonemptiness_Frechet_qsp},  this already yields
	$\liminf_{k\to\infty}D_k\geq 0$ and the statement of the theorem has been shown.
\end{proof}

In the subsequent remark, which parallels \cref{rem:Frechet_q10_vs_PMP,rem:Frechet_qs0_minimization},
we comment on necessary optimality conditions for unconstrained optimization problems involving
the functional $q_{s,p}$ as a sparsity-promoting term.
\begin{remark}\label{rem:Frechet_qsp_minimization}
	Fix $s\in[1,\infty)$, a Fr\'{e}chet differentiable function $f\colon L^s(\Omega)\to\R$, and
	consider the unconstrained minimization of $f+q_{s,p}$. A necessary condition for some
	$\bar u\in L^s(\Omega)$ to be a local minimizer of $f+q_{s,p}$ is given by 
	$-f'(\bar u)\in\fsub q_{s,p}(\bar u)$.
	\begin{enumerate}
		\item In case $s=1$, \cref{thm:characterization_Frechet_q1p} shows that this amounts to
			$\bar u=0$ and $f'(\bar u)=0$ almost everywhere on $\Omega$. A similar result can be
			obtained applying Pontryagin's maximum principle, 
			see \cite[Theorem~2.2]{ItoKunisch2014} or \cite[Section~2.1]{NatemeyerWachsmuth2020}.
		\item In case $s>1$, \cref{thm:characterization_Frechet_qsp} implies that 
			$f'(\bar u)\in L^r(\Omega)$ has to equal $p \abs{\bar u}^{p-2}\bar u$ almost everywhere
			on $\{\bar u\neq 0\}$.  
			This implicitly demands $\abs{\bar u}^{p-1}\chi_{\{\bar u\neq 0\}}\in L^r(\Omega)$ which
			promotes sparse controls since $\bar u$ has to approach zero from $\{\bar u\neq 0\}$
			if at all slowly enough.
	\end{enumerate}
\end{remark}

\subsection{Limiting subdifferential}

Let us now turn our attention to the limiting subdifferential constructions in
the Aspund space setting $s\in(1,\infty)$. 
Thanks to \cref{thm:characterization_Frechet_qsp}, we can adapt most
of the proof strategies directly from \cref{sec:limiting_p=0}.

\begin{theorem}\label{thm:characterization_limiting_qsp}
	Fix $s\in(1,\infty)$. Then we have
	\[
		\forall\bar u\in L^s(\Omega)\colon\quad
		\partial q_{s,p}(\bar u)
		=
		\fsub q_{s,p}(\bar u)
		=
		\{
			\eta\in L^r(\Omega)\,|\,
			\eta=p \abs{\bar u}^{p-2}\bar u
			\text{ a.e.\ on }\{\bar u\neq 0\}
		\}.
	\]
\end{theorem}
\begin{proof}
	The second ``$=$'' follows from \cref{thm:characterization_Frechet_qsp}
	and it remain to verify the first equality.
	The inclusion ``$\supset$'' follows
	from the definition of the limiting subdifferential.

	The proof of the converse inclusion ``$\subset$'' can be directly transferred from
	the one of \cref{thm:characterization_limiting_qs0} exploiting the different pointwise
	characterization of the Fr\'{e}chet subdifferential from \cref{thm:characterization_Frechet_qsp}.
\end{proof}

Next, we characterize the singular subdifferential of $q_{s,p}$.
\begin{theorem}\label{thm:characterization_singular_qsp}
	Fix $s\in(1,\infty)$. Then we have
	\[
		\forall\bar u\in L^s(\Omega)\colon\quad
		\partial^\infty q_{s,p}(\bar u)
		=
		\{
			\eta\in L^r(\Omega)\,|\,\{\eta\neq 0\}\subset \{\bar u= 0\}
		\}.
	\]
\end{theorem}
\begin{proof}
	Fix $\bar u\in L^s(\Omega)$. 
	Observe that in case where $\bar u$ vanishes almost everywhere on $\Omega$,
	we have $\fsub q_{s,p}(\bar u)=L^r(\Omega)$ from
	\cref{thm:characterization_Frechet_qsp} which, by definition of the
	singular subdifferential, already yields $\partial^\infty q_{s,p}(\bar u)=L^r(\Omega)$.
	Thus, we may assume throughout the remainder of the proof that $\{\bar u\neq 0\}$
	possesses positive measure.
	
	In order to prove the inclusion ``$\supset$'', we fix $\eta\in L^r(\Omega)$ which satisfies
	$\{\eta\neq 0\}\subset\{\bar u=0\}$.
	We set $\Omega_k:=\{\abs{\bar u}\geq 1/k\}$ as well as $u_k:=\bar u\chi_{\Omega_k}$
	for each $k\in\N$ leading to $u_k\to \bar u$ in $L^s(\Omega)$ and $q_{s,p}(u_k)\to q_{s,p}(\bar u)$,
	see \cref{lem:uniform_continuity_of_qsp}.
	For each $k\in\N$, let us define a measurable function $\eta_k\colon\Omega\to\R$ by means of
	\[
		\forall x\in\Omega\colon\quad
		\eta_k(x)
		:=
		\begin{cases}
			p|\bar u(x)|^{p-2}\bar u(x)	&\text{if $x\in\Omega_k$,}\\
			k\,\eta(x)			&\text{otherwise.}
		\end{cases}
	\]
	For each $k\in\N$, we find
	\begin{align*}
		\norm{\eta_k}_r^r
		&=
		\int_{\{\abs{\bar u}\geq 1/k\}}p^r|\bar u(x)|^{(p-1)r}\,\dx
		+
		\int_{\{\abs{\bar u}<1/k\}}k^r|\eta(x)|^r\,\dx\\
		&
		\leq
		p^rk^{(1-p)r}\blambda(\{\bar u\neq 0\})+k^r\norm{\eta}_r^r
		<
		\infty,
	\end{align*}
	i.e., $\{\eta_k\}_{k\in\N}\subset L^r(\Omega)$.
	Furthermore, $\eta_k\in\fsub q_{s,p}(u_k)$ follows from \cref{thm:characterization_Frechet_qsp}
	since $\{u_k\neq 0\}=\Omega_k$ is valid for each $k\in\N$.
	Noting that $\eta$ vanishes on $\Omega_k$, we obtain
	\begin{align*}
		\norm{\tfrac1{k}\eta_k-\eta}_r^r
		&=
		\frac{p^r}{k^{r}}\int_{\Omega_k}|\bar u(x)|^{(p-1)r}\,\dx
		\leq 
		\frac{p^r}{k^{r}}\int_{\Omega_k}k^{(1-p)r}\,\dx
		\leq
		\frac{p^r\,\blambda(\{\bar u\neq 0\})}{k^{pr}}
	\end{align*}
	for each $k\in\N$, and this shows $\tfrac1k\eta_k\to\eta$ in $L^r(\Omega)$.
	Particularly, we find $\eta\in\partial^\infty q_{s,p}(\bar u)$ by definition of the
	singular subdifferential.
	
	In order to prove the inclusion ``$\subset$'', let us fix $\eta\in\partial^\infty q_{s,p}(\bar u)$.
	By definition of the singular subdifferential, we find sequences 
	$\{u_k\}_{k\in\N}\subset L^s(\Omega)$, $\{\eta_k\}_{k\in\N}\subset L^r(\Omega)$, and
	$\{t_k\}_{k\in\N}\subset(0,\infty)$ such that $u_k\to\bar u$ in $L^s(\Omega)$,
	$t_k\searrow 0$, and $t_k\eta_k\weakly\eta$ in $L^r(\Omega)$ hold while
	$\eta_k\in\fsub q_{s,p}(u_k)$ is valid for each $k\in\N$.
	Along a subsequence (without relabeling), $\{u_k\}_{k\in\N}$ converges pointwise almost
	everywhere to $\bar u$. Thus, for almost all $x\in\{\bar u\neq 0\}$,
	we find $u_k(x)\to\bar u(x)\neq 0$, i.e., $x\in\{u_k\neq 0\}$ and, due to
	\cref{thm:characterization_Frechet_qsp}, $\eta_k(x)=p|u_k(x)|^{p-2}u_k(x)$ for sufficiently
	large $k\in\N$. Thus, for almost every $x\in\{\bar u\neq 0\}$, we have
	$t_k\,\eta_k(x)\to 0$. Thus, the weak convergence $t_k\eta_k\weakly\eta$ in
	$L^r(\Omega)$ ensures that $\eta$ needs to vanish almost everywhere on $\{\bar u\neq 0\}$,
	i.e., $\{\eta\neq 0\}\subset\{\bar u=0\}$.
\end{proof}

We would like to focus the reader's attention to the fact that the limiting subdifferential
$\partial q_{s,p}(\bar u)$ might be empty for some $\bar u\in L^s(\Omega)$ where $\{\bar u\neq 0\}$
possesses positive measure while $\abs{\bar u}^{p-1}\chi_{\{\bar u\neq 0\}}$ lacks of $L^r$-regularity.
In contrast, the singular subdifferential $\partial^\infty q_{s,p}(\bar u)$ has been shown
to be never empty.

We close the section by showing that $q_{s,p}$ is nowhere Lipschitz continuous.
\begin{corollary}\label{cor:Lipschitzness_qsp}
	For $s\in(1,\infty)$, $q_{s,p}$ is nowhere Lipschitz continuous. 
\end{corollary}
\begin{proof}
	For large parts, the proof parallels the one of \cref{cor:Lipschitzness_qs0}.
	Again, the situation is easy whenever $\bar u\in L^s(\Omega)$ satisfies
	$\blambda(\{\bar u=0\})>0$ due to 
	\cref{lem:Lipschitzness_s>1} and \cref{thm:characterization_singular_qsp}.
	Thus, we assume that $\{\bar u=0\}$ is of measure zero and show that $q_{s,p}$
	fails to satisfy the condition from \cref{lem:Lipschitzness_s>1}~\ref{item:SNEC} at $\bar u$.
	Therefore, we first choose $\alpha>0$ such that $\{\abs{\bar u}\geq\alpha\}$
	is of positive measure, define $\Omega_k:=\{\abs{\bar u}\geq\alpha/k\}$
	for each $k\in\N$, and pick a subset $\Omega_k'\subset\Omega_k$ of positive measure
	for each $k\in\N$ such that $\blambda(\Omega_k')\searrow 0$ holds.
	For each $k\in\N$, we define $u_k:=\bar u\chi_{\Omega_k\setminus\Omega_k'}$.
	We find
	$u_k\to\bar u$ in $L^s(\Omega)$, and
	\cref{lem:uniform_continuity_of_qsp} guarantees $q_{s,p}(u_k)\to q_{s,p}(\bar u)$.
	For each $k\in\N$, we define $\eta_k\in L^r(\Omega)$ by means of
	\[
		\forall x\in\Omega \colon\quad
		\eta_k(x)
		:=
		\begin{cases}
			\tfrac pk|\bar u(x)|^{p-2}\bar u(x)	&	\text{if $x\in\Omega_k\setminus\Omega_k'$,}\\
			\blambda(\Omega_k')^{-1/r}			&	\text{if $x\in\Omega_k'$,}\\
			0									&	\text{otherwise.}
		\end{cases}
	\]
	By construction, we have $\norm{\eta_k}_r\geq 1$ for each $k\in\N$.
	On the other hand, for each $h\in L^s(\Omega)$ and $k\in\N$, we find
	\begin{align*}
		\left|\int_\Omega \eta_k(x)h(x)\,\dx\right|
		&\leq
		\frac{p}{k}\int_{\Omega_k\setminus\Omega_k'}|\bar u(x)|^{(p-2)}\bar u(x)h(x)\,\dx
		+
		\norm{h}_{s,\Omega_k'}\\
		&\leq
		\frac{p}{\alpha^{1-p}k^p}\blambda(\{\bar u\neq 0\})^{1/r}
		\norm{h}_{s,\{\bar u\neq 0\}}
		+
		\norm{h}_{s,\Omega_k'}
	\end{align*}
	from H\"older's inequality on $\Omega_k\setminus\Omega_k'$ and $\Omega_k'$,
	respectively,
	and this yields $\eta_k\weakly 0$ in $L^r(\Omega)$.
	Due to $\eta_k\in\tfrac1k\fsub q_{s,p}(u_k)$ for each $k\in\N$, see
	\cref{thm:characterization_Frechet_qsp}, this shows that
	$q_{s,p}$ cannot be Lipschitz continuous at $\bar u$,
	see \cref{lem:Lipschitzness_s>1}.
\end{proof}

\section{Concluding remarks}\label{sec:concluding_remarks}

In this paper, we derived exact formulas for the Fr\'{e}chet, limiting, and singular subdifferential
of the functional $q_{s,p}$ 
defined in \eqref{eq:definition_of_sparsity_functional_p>0} and
\eqref{eq:definition_of_sparsity_functional_p=0}.
As \cref{rem:Frechet_q10_vs_PMP,rem:Frechet_qs0_minimization,rem:Frechet_qsp_minimization} underline,
the formulas for the Fr\'{e}chet subdifferential can be used in order to derive necessary optimality
conditions for the unconstrained minimization of functions $f+q_{s,p}$ where $f\colon L^s(\Omega)\to\R$
is Fr\'{e}chet differentiable. Let us now assume that $f+q_{s,p}$ has to be minimized 
with respect to some
constraint set $U_\textup{ad}\subset L^s(\Omega)$. Then Fermat's rule yields
validity of $0\in\fsub(f+q_{s,q}+\delta_{U_\textup{ad}})(\bar u)$ for each associated local
minimizer $\bar u\in L^s(\Omega)$ of the problem where 
the so-called indicator function 
$\delta_{U_\textup{ad}}\colon L^s(\Omega)\to\R\cup\{\infty\}$ of $U_\textup{ad}$
equals $0$ on $U_{\textup{ad}}$ and is set to $\infty$, otherwise.
Note that the Fr\'{e}chet subdifferential does not obey a sum rule as soon as not all but one addends
are smooth.
In the present situation, the simultaneous non-Lipschitzness of $q_{s,p}$ and $\delta_{U_\textup{ad}}$ 
does not even allow to apply the \emph{fuzzy} sum rule of Fr\'{e}chet subdifferential calculus, see
\cite[Theorem~2.33]{Mordukhovich2006}, and take the limit afterwards. 
Thus, one may try to evaluate the slightly weaker necessary
optimality condition $0\in\partial(f+q_{s,q}+\delta_{U_\textup{ad}})(\bar u)$ since the sum rule
for the limiting subdifferential applies to non-Lipschitz functions as well, see 
\cite[Theorem~3.36]{Mordukhovich2006}. Unluckily, this will not be a straight task since 
both of the functionals $q_{s,p}$ and $\delta_{U_\textup{ad}}$ are not so-called 
sequentially normally epi-compact on their respective domains, see comments at the end of
\cref{sec:variational_analysis} and the proofs of \cref{cor:Lipschitzness_qs0,cor:Lipschitzness_qsp}.
Nevertheless, for particular choices of $U_\textup{ad}$ like box-constrained sets,
there might be a chance to show validity of the sum rule by inherent problem structure and, thus,
obtain necessary optimality conditions in terms of the limiting subdifferential.

\section*{Acknowledgments}
This work is supported by the DFG Grant \emph{Bilevel Optimal Control: Theory, Algorithms, and Applications}
(Grant No.\ WA 3636/4-2) within the Priority Program SPP 1962 (Non-smooth
and Complementarity-based Distributed Parameter Systems: Simulation and Hierarchical Optimization).


\end{document}